\documentclass[paper=a4,headsepline,DIV13]{scrartcl}

\usepackage[utf8]{inputenc}
\usepackage{amsmath}
\usepackage{amssymb}
\usepackage{amsthm}
\usepackage{enumerate}
\usepackage{graphicx}
\usepackage{booktabs}
\usepackage{subfig}
\usepackage{color}
\usepackage{mathtools}
\usepackage{subfig}
\usepackage{hyperref}
\makeatletter
\newcommand{\customlabel}[2]{%
   \protected@write \@auxout {}{\string \newlabel {#1}{{#2}{\thepage}{#2}{#1}{}} }%
   \hypertarget{#1}{#2}
}
\makeatother

\usepackage[backend=biber]{biblatex}
\renewenvironment{abstract}{\minisec{Abstract}}{\par\vspace{.1in}}
\newenvironment{keywords}{\minisec{Key Words}}{\par\vspace{.1in}}
\newenvironment{AMS}{\minisec{AMS subject classification}}{\par\vspace{.1in}}

\theoremstyle{plain}
\newtheorem{theorem}{Theorem}[section]
\newtheorem{corollary}[theorem]{Corollary}
\newtheorem{lemma}[theorem]{Lemma}
\newtheorem{assumption}[theorem]{Assumption}

\theoremstyle{remark}

\usepackage{pst-node,pst-eucl,pstricks-add}
\usepackage[off,crop=on]{auto-pst-pdf}

\newcommand{\TheTitle}{Error estimates for the finite element approximation of bilinear boundary control problems}

\newcommand{\lnh}{\lvert\ln h\rvert}

\title{\TheTitle}
\author{Max Winkler\thanks{Chemnitz University of Technology, Faculty of
  Mathematics, Professorship Numerical Mathematics (Partial Differential
  Equations), Stra\ss e der Nationen 62, 09111 Chemnitz, Germany. max.winkler@mathematik.tu-chemnitz.de}
}
\date{April 16, 2018}
\DeclareMathOperator{\sgn}{sgn}

\DeclareMathOperator{\argmax}{arg\,max}

\begin{document}

\maketitle

\begin{abstract}
  In this article a special class of nonlinear optimal control problems
  involving a bilinear term in the boundary condition is studied. These kind of problems
  arise for instance in the identification of an unknown space-dependent Robin
  coefficient from a given measurement of the state, or when the Robin
  coefficient can be controlled in order to reach a desired state.
  To this end, necessary and sufficient optimality conditions are derived and
  several discretization approaches for the numerical solution
  the optimal control problem are investigated.
  Considered are both a full discretization and the postprocessing approach
  meaning that we compute an improved control by a pointwise evaluation of the
  first-order optimality condition. For both approaches finite element error
  estimates are shown and the validity of these results is confirmed by numerical experiments.
  \keywords{Bilinear boundary control, identification of Robin
    parameter,
    finite element error estimates,
    postprocessing approach}

  \AMS{35J05, 35Q93, 49J20, 65N21, 65N30}
\end{abstract}

\section{Introduction}

This paper is concerned with bilinear boundary control problems of the form
\begin{equation*}
\frac12 \|y-y_d\|_{L^2(\Omega)}^2 + \frac\alpha2\|u\|_{L^2(\Gamma)}^2 \to \min!
\end{equation*}
subject to
\begin{equation*}
  \begin{aligned}
    -\Delta y + y &= f &&\mbox{in}\ \Omega,\\
    \partial_n y + u\,y &= g && \mbox{on}\ \Gamma,
  \end{aligned}
\end{equation*}
\begin{equation*}
  u\in U_{ad}:=\{v\in L^2(\Gamma)\colon u_a \le u \le u_b\ \mbox{a.\,e.\ on}\ \Gamma\},
\end{equation*}
where $\Omega\subset\mathbb R^n$, $n\in\{2,3\}$, is a bounded domain, $\alpha>0$ is the regularization parameter, $y_d\in L^2(\Omega)$ is a desired state
and $0 \le u_a < u_b$ are the control bounds.

As an application of bilinear boundary control problems
we mentioned the identification of an unknown Robin coefficient
from a given measurement $y_d$ of the state quantity.
This is for instance of interest in the modeling of stem cell division processes \cite{FRT16},
where $u$ is the unknown parameter describing the chemical reactions between proteins
from the cell interior and the cell cortex. For further applications,
$u$ can be interpreted as a heat-exchange
coefficient in thermodynamics or as a quantity for corrosion damage
in electrostatics.
There are many publications dealing with the identification of the Robin
coefficient, see for instance \cite{Chaabane1999,Hetmaniok2017,LiuNakamura2017,MohebbiSellier2018}.
Only a few papers use an optimal control approach similar to the one  considered in the
present article.
We mention \cite{JL12,HaoThanhLesnic2013}, where the parabolic version of our model problem
is considered.
The authors prove convergence of a finite element approximation but no
convergence rate is established. A similar problem is discussed in 
\cite{Gwinner2018}, dealing with the recovery of the Robin parameter in a variational
inequality.

The aim of the present paper is to derive necessary and sufficient optimality conditions
for the optimal control problem and to investigate several numerical approximations
regarding convergence towards a local solution.
This complements a previous contribution of Kr\"oner and Vexler
\cite{KV09} where the distributed control case, meaning that the bilinear term
$u\,y$ appears in the differential equation, is discussed. In this article
error estimates for the approximate controls in the $L^2(\Omega)$-norm are derived 
for several finite element approximations. To this end, the authors approximate the
state and adjoint state by linear finite elements and 
derive the convergence rate $1$ for piecewise constant and $3/2$ for piecewise linear
approximations for the control.  Moreover, advanced discretization concepts like
the postprocessing approach \cite{MR11} and
the variational discretization \cite{Hin05} is investigated which allow an improvement
up to a convergence rate of $2$. It is the purpose of the present article to extend the results
to the case of boundary control.

The numerical analysis of boundary control problems 
is usually more difficult than for distributed control problems as the adjoint control--to--state
operator maps onto some Sobolev/Lebes\-gue space defined on the boundary.
As a consequence, error estimates for the traces of finite element solutions
have to be proved, more precisely, in the $L^2(\Gamma)$-norm.
Here, we consider two different discretization approaches. The first one is a full
discretization using piecewise linear finite elements for the states and
piecewise constant functions on the boundary for the control approximation.
Under the assumption that the domain has a Lipschitz boundary we show that the
discrete optimal control converges with the optimal rate $1$. To show this result we
exploit the local coercivity of the objective, best-approximation properties
of the control space and suboptimal error estimates for the state and adjoint equation.
In order to obtain a more accurate solution we also investigate the
postprocessing approach
where an improved control is computed by a pointwise application of the
first-order optimality condition to the discrete state
variables. For this approach we have to assume more regularity for the exact
solution and thus, we restrict our considerations to two-dimensional domains
with sufficiently smooth boundary. Under this assumption we show the optimal
convergence rate of $2-\varepsilon$ with arbitrary $\varepsilon>0$
which is the rate one would
also expect in the case of linear quadratic boundary control problems
and smooth solutions \cite{APR13,APW17,MR04} (even with $h^{-\varepsilon}$
replaced by $\lnh$, where $h$ is the maximal element diameter of the finite
element mesh).
The proof relies on the non-expansivity of the projection onto the feasible
set as well as sharp error estimates for the state and adjoint state in
$L^2(\Gamma)$. To obtain estimates in these norms superconvergence properties
of the midpoint interpolant, finite element error estimates for the Ritz
projection in $L^2(\Gamma)$ and a supercloseness result between the
midpoint interpolant of the exact and the discrete solution are exploited.
To show the $L^2(\Gamma)$-norm error estimate we will, as we consider smooth solutions,
derive a maximum norm estimate.
To the best of the authors knowledge these results are not available in the literature
for problems with Robin boundary conditions. Based on the
ideas from \cite{FR76} we formulate the missing proof.

We moreover note that the setting discussed here does not fit into the
well-known framework of the semilinear optimal control problems discussed
e.\,g.\ in \cite{ACT02,CM08,CMT05,KP14}, as these contributions deal with nonlinearities
depending solely on the state variable. However, many techniques can be reused for the problem
considered here. The only publication where more
general nonlinearities depending both on 
the state and the control variable is, to the best of the authors knowledge,
\cite{RoeschTroeltzsch1992}. Therein optimality conditions are discussed
but there is no theory on the numerical analysis of approximation methods
for this problem class available yet. However, we think that the consideration of
bilinear control problems may serve as a starting point for the investigation of a more general
class of nonlinear optimal control problems.

The article is structured as follows. In Section~\ref{sec:state_equation} we
discuss the solubility of the state equation and regularity results for its solution.
In Section~\ref{sec:optimal_control} we analyze the optimal
control problem. In particular, necessary and sufficient optimality conditions
are investigated. Section \ref{sec:fem} is devoted to the finite element
discretization of the state equation, where we show finite element error estimates
required for the numerical analysis of the optimal control problem later.
The discretization of the optimal control problem is considered in
Section~\ref{sec:discrete_optimal_control}.
In particular, we discuss convergence
rates for the numerical solution obtained by a full discretization of the
optimal control problem as well as for an improved control obtained by a
postprocessing step. The latter result requires some auxiliary results that we
discuss in the appendix.
To be more precise, a maximum norm error estimate for the finite element solution
of an elliptic equation with Robin boundary conditions is needed. A proof is
given in Appendix \ref{app:maximum_norm_estimate}.
Moreover, a proof of local error estimates for
the midpoint interpolant and the $L^2(\Gamma)$ projection onto piecewise
constant functions on the boundary is needed.
To the best of the authors knowledge these results are not available in the
literature in case of domains with curved boundaries. Thus, we discuss these
auxiliary results in Appendix \ref{app:midpoint}.
Finally, we will compare the theoretical results with numerical experiments
in Section~\ref{sec:experiments}.

\section{Analysis of the state equation}
\label{sec:state_equation}

We consider the boundary value problem 
\begin{equation*}
 -\Delta y + y = f \ \mbox{in}\ \Omega,\qquad \partial_n y + u\,y=g\ \mbox{on}\ \Gamma,
\end{equation*}
on a bounded Lipschitz domain $\Omega\subset\mathbb R^n$, $n\in\{2,3\}$, with data $f\in L^2(\Omega)$ and $g\in L^2(\Gamma)$.
The corresponding weak formulation reads
\begin{equation}\label{eq:weak_form}
  \text{Find}\ y\in H^1(\Omega)\colon\qquad a_u(y,v) = F(v)\qquad\forall v\in H^1(\Omega),
\end{equation}
with
\begin{align*}
  a_u(y,v)&:=(\nabla y,\nabla v)_{L^2(\Omega)} + (y,v)_{L^2(\Omega)} + (u\,y,v)_{L^2(\Gamma)},\\
  F(v) &:= (f,v)_{L^2(\Omega)} + (g,v)_{L^2(\Gamma)}.
\end{align*}
First, we show an existence and uniqueness result for \eqref{eq:weak_form}.
Therefore, we introduce a decomposition of the control into a positive and negative part $u^+,u^-\in L_+^2(\Gamma):=\{v\in L^2(\Gamma)\colon v\ge 0\mbox{\ a.\,e.\ on}\ \Gamma\}$ 
such that $u=u^+ - u^-$. The following result then relies on the Lax-Milgram-Lemma. However,
an assumption on the coefficient $u$ is required.
\begin{lemma}\label{lem:lax_milgram}
  Assume that $u\in L^2(\Gamma)$ satisfies
  \begin{equation}\label{eq:ass_coercivity}
    \|u^-\|_{L^2(\Gamma)} < \frac1{c_*^2}.
  \end{equation}
  with the constant $c^*$ which is due to the embedding $\|v\|_{L^4(\Gamma)} \le c^*\|v\|_{H^1(\Omega)}$.
  Then, the solution $y$ of \eqref{eq:weak_form} belongs to $H^1(\Omega)$
  and satisfies the a priori estimate
  \begin{equation*}
    \|y\|_{H^1(\Omega)}\le \frac{1}{\gamma_u}\,\left(\|f\|_{H^1(\Omega)^*} + \|g\|_{H^{-1/2}(\Gamma)}\right).
  \end{equation*}
  with $\gamma_u := 1-c_*^2\,\|u^-\|_{L^2(\Gamma)}>0$.
\end{lemma}
\begin{proof}
  The boundedness of $a_u$ follows directly from the Cauchy-Schwarz inequality and
  the embedding $H^1(\Omega)\hookrightarrow L^4(\Gamma)$. This implies.
  \begin{align*}
    a(y,z) &\le \|y\|_{H^1(\Omega)}\,\|z\|_{H^1(\Omega)}
             + \|u\|_{L^2(\Gamma)}\,\|y\|_{L^4(\Gamma)}\, \|z\|_{L^4(\Gamma)} \\
           &\le \left(1+c_*^2\,\|u\|_{L^2(\Gamma)}\right)\,\|y\|_{H^1(\Omega)}\, \|z\|_{H^1(\Omega)}.
  \end{align*}
  To show the coercivity we take into account the decomposition $u=u^+ - u^-$
  and the embedding $H^1(\Omega)\hookrightarrow L^4(\Gamma)$ to get
  \begin{align*}
    a(y,y) & \ge \|y\|_{H^1(\Omega)}^2 + \int_\Gamma u^-\,y^2 \ge \left(1-c_*^2\,\|u^-\|_{L^2(\Gamma)}\right)\,\|y\|_{H^1(\Omega)}^2.
  \end{align*}
  Here, the assumption \eqref{eq:ass_coercivity} will ensure the coercivity.
  An application of the Lax-Milgram Lemma leads to the desired result.
\end{proof}
Note that $\{v\in L^2(\Gamma)\colon \|v^-\|_{L^2(\Omega)} < c_*^{-2}\}$
is an open subset of $L^2(\Gamma)$. This is the key idea which allows us to
avoid the two-norm discrepancy for the optimal control problem
as we will see that the reduced objective functional is differentiable with respect
to the $L^2(\Gamma)$-topology. In the following we will hide the dependency of the estimates
on $\|u^-\|_{L^2(\Gamma)}$ and thus $\gamma_u$ in the generic constant as we impose
positive control bounds in the considered optimal control problem.

Later, we will frequently make use of the following Lipschitz estimate.
\begin{lemma}\label{lem:lipschitz_general}
  If $u_1,u_2\in L^2(\Gamma)$ satisfy the assumption \eqref{eq:ass_coercivity},
  the corresponding states $y_1,y_2\in H^1(\Omega)$ solving 
  \begin{equation*}
    a_{u_i}(y_i,v) = (f_i,v) + (g_i,v)_\Gamma  \quad\forall v\in H^1(\Omega),\ i=1,2,
  \end{equation*}
  fulfill the estimate
  \begin{align*}
    \|y_1-y_2\|_{H^1(\Omega)} &\le \big(\|u_1-u_2\|_{L^2(\Gamma)}\,\|y_2\|_{H^1(\Omega)} \\
    &\quad + \|f_1-f_2\|_{H^1(\Omega)^*}
      + \|g_1-g_2\|_{H^{-1/2}(\Gamma)}\big).
  \end{align*}
\end{lemma}
\begin{proof}
  Subtracting the variational formulations for $y_1$ and $y_2$ from each other leads to
  \begin{align*}
    & (\nabla(y_1-y_2),\nabla v)_{L^2(\Omega)} + (y_1-y_2,v)_{L^2(\Omega)} + (u_1\,(y_1-y_2),v)_{L^2(\Gamma)} \\
    &\qquad= (f_1-f_2,v)_{L^2(\Omega)} + (g_1-g_2,v)_{L^2(\Gamma)} + ((u_2-u_1)\,y_2,v)_{L^2(\Gamma)}.
  \end{align*}
  The result follows from Lemma \ref{lem:lax_milgram} and the continuity of
  the product mapping from $L^2(\Gamma)\times H^{1/2}(\Gamma)$ to $H^{-1/2}(\Gamma)$, see
  \cite[Theorem 1.4.4.2]{Gri85}.  
\end{proof}
In the following theorem we collect some regularity results for the solution of \eqref{eq:weak_form}.
\begin{lemma}\label{lem:props_S}
  Let $\Omega\subset\mathbb R^n$, $n\in\{2,3\}$, be a bounded Lipschitz
  domain. By $y\in H^1(\Omega)$ we denote the solution of \eqref{eq:weak_form}.  
  The following a priori estimates are valid, under the assumption that the
  input data possess the regularity demanded by the right-hand side:%
  \begin{align*}
    \intertext{a) If  $r>2n/(1+n)$ and $p > 2$ for $n=2$ and $p\ge 4$ for $n=3$, then} 
    \|y\|_{H^{3/2}(\Omega)} + \|y\|_{H^1(\Gamma)}
    &\le c\left(1+\|u\|_{L^p(\Gamma)}\right)\left( \|f\|_{L^{r}(\Omega)} + \|g\|_{L^2(\Gamma)}
     \right).\\
        \intertext{b) If  $r > n/2,\ s > n-1,$ and $p\ge2$ for $n=2$ and $p > 8/3$ for $n=3$, then}
    \|y\|_{C(\overline\Omega)}
    &\le c \left(1+\|u\|_{L^{p}(\Gamma)}\right)^2\left(\|f\|_{L^{r}(\Omega)}
      + \|g\|_{L^{s}(\Gamma)}\right). \\
    \intertext{c) Furthermore, if $\Omega$ is a convex polygonal/polyhedral
    domain, or possesses a boundary which is of class $C^{1,1}$, there holds}
    \|y\|_{H^2(\Omega)} &\le c\left(1+\|u\|_{H^{1/2}(\Gamma)}\right)^2\,\left(\|f\|_{L^2(\Omega)} + \|g\|_{H^{1/2}(\Gamma)}\right).
  \end{align*}
\end{lemma}
\begin{proof}
  a)
  In \cite[Theorem 1.12]{Dha12} it is shown that the problem
  \begin{equation*}
    -\Delta y = F\ \mbox{in}\ \Omega,\qquad \partial_n y = G \ \mbox{on}\ \Gamma
  \end{equation*}
  possesses a solution in $H^{3/2}(\Omega)$ provided that
  $F\in H^{s-2}(\Omega)$ for some $s\in (3/2,2]$ and $G\in L^2(\Gamma)$, as well as
  $\int_\Omega F + \int_\Gamma G = 0$. The solubility condition
  is satisfied in our situation with $F=f-y$ and $G=g-u\,y$ and becomes clear when testing
  \eqref{eq:weak_form} with $v\equiv 1$. The regularity required for $F$ follows from the embedding
  $f\in L^r(\Omega)\hookrightarrow H^{-1/2+\varepsilon}(\Omega)$ for sufficiently small
  $\varepsilon>0$.
  Moreover, the H\"older inequality and the embeddings
  $H^1(\Omega)\hookrightarrow L^q(\Gamma)$ with $q<\infty$ ($n=2$)
  or $H^1(\Omega)\hookrightarrow L^4(\Gamma)$ ($n=3$)
  imply $\|u\,y\|_{L^2(\Gamma)} \le c\,\|u\|_{L^p(\Omega)}\,\|y\|_{H^1(\Omega)}$, from which we conclude
  $G\in L^2(\Gamma)$. From \cite[Theorem 1.12]{Dha12} and Lemma \ref{lem:lax_milgram}
  we then obtain
  \begin{align}\label{eq:a_priori_h32}
    \|y\|_{H^{3/2}(\Omega)}
    &\le c \left( \|F\|_{L^r(\Omega)} + \|G\|_{L^2(\Gamma)} + \left\vert\int_\Omega y(x)\mathrm dx\right\vert\right) \nonumber\\
    &\le c \left(1+\|u\|_{L^p(\Gamma)}\right)\left(\|f\|_{L^r(\Omega)} + \|g\|_{L^2(\Gamma)}\right).
  \end{align}
        
  It remains to show the $H^1(\Gamma)$-norm estimate.
  We split the solution into the parts $y_f$ and $y_g$ solving
  \begin{equation*}
  \begin{aligned}
    -\Delta y_f + y_f &= f &\qquad -\Delta y_g + y_g &= 0 &\qquad& \mbox{in}\ \Omega,\\
    \partial_n y_f &= 0 & \partial_n y_g &= g - u y && \mbox{on}\ \Gamma.
  \end{aligned}
\end{equation*}
  Using \cite[Theorem 5.4]{GM11}  we directly deduce
  \begin{equation*}
    \|y_g\|_{H^1(\Gamma)} \le c\,\|g-u\,y\|_{L^2(\Gamma)} \le c \left(\|g\|_{L^2(\Gamma)} + \|u\|_{L^p(\Gamma)}\,\|y\|_{H^1(\Omega)}\right)
  \end{equation*}
  and Lemma \ref{lem:lax_milgram} leads to the desired estimate for $y_g$.
  For the function $y_f$, we get the desired estimate by an application of a trace theorem and the
  a priori estimate \eqref{eq:a_priori_h32} which can in case of $g\equiv 0$ be improved to
  \begin{equation*}
    \|y_f\|_{H^1(\Gamma)} \le c\,\|y_f\|_{H^{3/2+\varepsilon}(\Omega)}
    \le c\,\|f\|_{L^r(\Omega)},
  \end{equation*}
  provided that $\varepsilon >0$ is sufficiently small. The validity of the second step
  can be confirmed by means of \cite[Theorem 1.12]{Dha12} and \cite[Theorem 23.3]{Dau88}.
  The decomposition $y=y_f+y_g$ and the estimates shown above imply the desired estimate in
  the $H^1(\Gamma)$-norm.
  
  b) We prove the result for the case $n=3$. The two-dimensional case follows from the same arguments.
  From \cite[Theorem 3.1]{Cas93} it is known that the solution of \eqref{eq:weak_form}
  belongs to $C(\overline\Omega)$ if $f\in L^r(\Omega)$, $r>n/2$, and $g-uy\in L^s(\Gamma)$, $s>n-1$.
  The latter assumption can be concluded from the H\"older inequality, a Sobolev embedding and a trace theorem, which implies
  \[
    \|u\,y\|_{L^s(\Gamma)} \le c\,\|u\|_{L^p(\Gamma)}\,\|y\|_{L^8(\Gamma)}\le c\,\|u\|_{L^p(\Gamma)}\,\|y\|_{H^{5/4+\varepsilon}(\Omega)}
  \]
  for $1/p+1/8=1/(2+\varepsilon)$. A simple computation shows that $p>8/3$
  and $s=2+\varepsilon$ with $\varepsilon>0$ sufficiently small guarantee the validity of
  the previous steps.
  It remains to show $y\in H^{5/4+\varepsilon}(\Omega)$. This can be deduced from
  \cite[Theorem 23.3]{Dau88} where the a priori estimate  
  \begin{equation}\label{eq:h54_reg}
    \|y\|_{H^{5/4+\varepsilon}(\Omega)} \le c \left(\|f\|_{H^{3/4-\varepsilon}(\Omega)^*} + \|g - u\,y\|_{H^{-1/4+\varepsilon}(\Gamma)}\right)
  \end{equation}
  is stated.
  The regularity demanded by the right-hand side of \eqref{eq:h54_reg} is confirmed with the
  embeddings $f\in L^r(\Omega)\hookrightarrow H^{3/4-\varepsilon}(\Omega)^*$ and
  $g\in L^s(\Gamma)\hookrightarrow H^{-1/4+\varepsilon}(\Gamma)$.
  Moreover, there holds $\|u\,y\|_{H^{-1/4+\varepsilon}(\Gamma)} \le c\,\|u\|_{L^p(\Gamma)}\,\|y\|_{L^4(\Gamma)}$, see \cite[Theorem 1.4.4.2]{Gri85}.
  Collecting up the arguments above leads to
  \begin{align*}
    \|y\|_{C(\overline\Omega)}
    &\le c \left(\|f\|_{L^r(\Omega)} + \|g\|_{L^s(\Gamma)} + \|u\|_{L^p(\Gamma)} \|y\|_{H^{5/4+\varepsilon}(\Omega)}\right) \\
    &\le c \left(1+\|u\|_{L^p(\Gamma)}\right)^2\left(\|f\|_{L^r(\Omega)} + \|g\|_{L^s(\Omega)}
      + \|y\|_{H^1(\Omega)}\right)
  \end{align*}
  and the assertion follows after insertion of the a priori estimate from Lemma \ref{lem:lax_milgram}.

  c)
  With an embedding we deduce from the assumption that $u\in L^4(\Gamma)$.
  Hence, \eqref{eq:h54_reg} is applicable which implies $y\in H^{3/4}(\Gamma)$ and thus, $u\,y\in H^{1/2}(\Gamma)$, see  \cite[Theorem 1.4.4.2]{Gri85}.
  The $H^2(\Omega)$-regularity of $y$ then follows from a shift theorem applied to the
  equation with boundary conditions $\partial_n y = g - u y\in H^{1/2}(\Gamma)$ on $\Gamma$,
  see \cite[Theorem 2.4.2.7]{Gri85} (for domains with smooth boundary) or \cite[Theorem 4.4.3.8]{Gri85} (for convex polygonal domains).
\end{proof}

\section{The optimal control problem}
\label{sec:optimal_control}

We introduce the control-to-state operator $S\colon U_{ad}\to H^1(\Omega)$
defined by $S(u):=y$, with $y$ the solution of \eqref{eq:weak_form}.
In this section we discuss the bilinear optimal control problem
\begin{equation}\label{eq:target}
  j(u):=\frac12\|S(u)-y_d\|_{L^2(\Omega)}^2 + \frac\alpha2\|u\|_{L^2(\Gamma)}^2 \to \min!
\end{equation}
subject to $u\in U_{ad}:=\{v\in L^2(\Gamma)\colon u_a\le v\le u_b\ \text{a.\,e.\ on}\ \Gamma\}$.
Here, $\alpha>0$ is the regularization parameter, $y_d\in L^2(\Omega)$ the desired state
and $0 < u_a < u_b$ the control bounds. 
Our aim is to derive necessary and sufficient optimality conditions as
well as regularity results for local solutions. Note, that the operator $S$ is non-affine and
consequently, $j$ is non-convex.

\subsection{Optimality conditions}

To derive optimality conditions differentiability properties of the (implicitly defined) operator $S$ are of interest.
\begin{lemma}\label{lem:differentiability}
 The operator $S\colon U_{ad}\to H^1(\Omega)$ is infinitely many times Fr\'echet differentiable with respect to the $L^2(\Gamma)$-topology.
 The first derivative $\delta y := S'(u)\delta u$ is the weak solution of
 \begin{equation}\label{eq:tangent_equation}
   \left\lbrace
     \begin{array}{rlll}
       -\Delta \delta y + \delta y &= 0 &\qquad&\mbox{in}\ \Omega,\\
       \partial_n\delta y + u\,\delta y &= -\delta u\,y &&\mbox{on}\ \Gamma.
     \end{array}
   \right.
 \end{equation}
\end{lemma}
\begin{proof}
 The result follows from an application of the implicit function theorem to
 the operator $e\colon H^1(\Omega)\times U  \to H^1(\Omega)^*$ with
 $U:=\{v\in L^2(\Gamma)\colon v\mbox{ fulfills \eqref{eq:ass_coercivity}}\}$ defined by
 \begin{equation*}
   e(y,u)v := (\nabla y,\nabla v)_{L^2(\Omega)} + (y,v)_{L^2(\Omega)} + (u\,y,v)_{L^2(\Gamma)} - (f,v)_{L^2(\Omega)} - (g,v)_{L^2(\Omega)},
 \end{equation*}
 whose roots are solutions of \eqref{eq:weak_form}.
 We choose $\delta y\in H^1(\Omega)$, $\delta u\in U$ such that
 $u+\delta u\in U$ (note that $U$ is an open subset of $L^2(\Gamma)$).
 First, we confirm that the linear operator $e'(y,u)\colon H^1(\Omega)\times U\to H^1(\Omega)^*$
 defined by
 \begin{equation}\label{eq:proof_derivative_S}
  e'(y,u)(\delta y, \delta u):=(\nabla \delta y,\nabla \cdot)_{L^2(\Omega)} + (\delta y,\cdot)_{L^2(\Omega)} + (u\,\delta y + y\,\delta u,\cdot)_{L^2(\Gamma)}
 \end{equation}
 is the Fr\'echet-derivative of $e$. This is a consequence of
 \begin{equation*}
  e(y+\delta y,u+\delta u) - e(y,u) = e'(y,u)(\delta y,\delta u) + (\delta u\,\delta y,\cdot)_{L^2(\Gamma)}
 \end{equation*}
 and the fact that the remainder term satisfies
 \begin{align}
   \label{eq:remainder_term}
   \|\delta u\,\delta y\|_{H^1(\Omega)^*}
   &= \sup_{\varphi\in H^1(\Omega)} \frac{(\delta u\,\delta y,\varphi)_{L^2(\Gamma)}}{\|\varphi\|_{H^1(\Omega)}}
     \le c \sup_{\varphi\in H^1(\Omega)} \frac{\|\varphi\|_{L^4(\Gamma)}}{\|\varphi\|_{H^1(\Omega)}}\,
     \|\delta u\|_{L^2(\Gamma)}\,\|\delta y\|_{L^4(\Gamma)}  \nonumber\\
   &\le c\,\|\delta u\|_{L^2(\Gamma)}\,\|\delta y\|_{H^1(\Omega)}
     \le c\left(\|\delta u\|_{L^2(\Gamma)}^2+\|\delta y\|_{H^1(\Omega)}^2\right) \nonumber\\
   &= o(\|(\delta y,\delta u)\|_{H^1(\Omega)\times L^2(\Gamma)}),
 \end{align}
 where we applied the generalized H\"older inequality and the embedding $H^1(\Omega)\hookrightarrow L^4(\Gamma)$.
 The second Fr\'echet derivative $e''\colon H^1(\Omega)\times U\to \mathcal L((H^1(\Omega)\times L^2(\Gamma))^2,H^1(\Omega)^*)$ is given by
 \begin{equation*}
  e''(y,u)(\delta y,\delta u)(\tau y,\tau u) := (\tau u\,\delta y + \delta u\,\tau y,\cdot)_{L^2(\Gamma)}
 \end{equation*}
 and the mapping $(y,u)\mapsto e''(y,u)$ is continuous.
 The derivatives of order $n\ge 3$ vanish.
 Hence, $e\colon H^1(\Omega)\times U\to H^1(\Omega)^*$ is of class $C^\infty$.
 
 Finally, due to Lemma \ref{lem:lax_milgram} we conclude that the linear mapping
 \[\delta y\mapsto e_y(y,u)\delta y = (\nabla \delta y,\nabla\cdot)_{L^2(\Omega)} + (\delta y,\cdot)_{L^2(\Omega)}+(u\,\delta y,\cdot)_{L^2(\Gamma)}\in H^1(\Omega)^*\]
 is bijective. The implicit function theorem implies the assertion and the derivative 
 $\delta y:=S'(u)\delta u$ is given by $e'(y,u)(\delta y,\delta u) = 0$.
 This corresponds to the weak formulation of \eqref{eq:tangent_equation}.
\end{proof}

From the chain rule and Lemma \ref{lem:differentiability} we directly conclude the following differentiability result
\begin{lemma}
 The functional $j\colon U_{ad}\to\mathbb R$ is infinitely many times Fr\'echet differentiable with respect to the $L^2(\Gamma)$-topology
 and the first derivative is given by
 \begin{equation}\label{eq:opt_cond_varineq}
  \left<j'(u),v\right> = (S(u)-y_d, S'(u)v)_{L^2(\Omega)} +
  \alpha\,(u,v)_{L^2(\Gamma)},
  \qquad v\in L^2(\Gamma).
 \end{equation}
 \end{lemma}
 The optimality condition can be simplified using the adjoint of the linearized
 control-to-state operator, this is, 
 \[S'(u)^*\colon H^1(\Omega)^*\to L^2(\Gamma),\qquad
   S'(u)^*v:= -[S(u)\,p]_\Gamma\] with $p\in H^1(\Omega)$ solving the adjoint equation
 \begin{equation}
   \label{eq:adjoint_eq}
   \left\lbrace
     \begin{array}{rlll}
       -\Delta p + p &= v &\qquad&\mbox{in}\ \Omega,\\
       \partial_n p + u\,p &= 0 &&\mbox{on}\ \Gamma.
     \end{array}
   \right.
 \end{equation}
 In the following we denote the control-to-adjoint mapping $u\mapsto p=S'(u)^*(S(u)-y_d)$ with
 by $Z\colon L^2(\Gamma)\to H^1(\Omega)$.
 Consequently, we can rewrite the optimality condition \eqref{eq:opt_cond_varineq} as
 \begin{align}\label{eq:opt_cond}
   &
   \begin{aligned}
     -\Delta y + y &= f\qquad  & -\Delta p + p &= y-y_d &\qquad & \mbox{in}\ \Omega,\\
     \partial_n y + u\,y &= g & \partial_n p + u\,p &= 0 && \mbox{on}\ \Gamma,
   \end{aligned}\\[.3em]
   &\hspace{1cm}\left(\alpha\,u - y\,p,v-u\right)_{L^2(\Gamma)} \ge 0\hspace{1.95cm}\mbox{for all}\ v\in U_{ad}.   \nonumber
 \end{align}
 The variational inequality is equivalent to the projection formula
 \begin{equation}\label{eq:proj_formula}
   u = \Pi_{ad}\left(\frac1\alpha [y\,p]_\Gamma\right)
 \end{equation}
 with $\Pi_{ad}$ the $L^2(\Gamma)$-projection onto $U_{ad}$.
 
 To compute the second derivative of $j$ we need the solution
 $\delta y:=S'(u)\delta u\in H^1(\Omega)$ of the \emph{tangent equation}
 \begin{equation*}
   \left\lbrace
     \begin{array}{rlll}
       -\Delta \delta y + \delta y &= 0 &\quad&\mbox{in}\quad\Omega,\\
       \partial_n\delta y + u\,\delta y &= - y\,\delta u &&\mbox{on}\quad\Gamma,
     \end{array}
     \right.
   \end{equation*}
   and the solution $\delta p:=Z'(u)\delta u\in H^1(\Omega)$ of the \emph{dual for Hessian equation}
 \begin{equation*}
   \left\lbrace
     \begin{array}{rlll}
       -\Delta \delta p + \delta p &= \delta y &\qquad& \mbox{in}\quad\Omega,\\
       \partial_n\delta p + u\,\delta p &= -p\,\delta u &&\mbox{on}\quad\Gamma.
     \end{array}
   \right.
 \end{equation*}
 Then, the reduced Hessian in the directions $\delta u,\tau u\in L^2(\Gamma)$ then reads
 \begin{equation}\label{eq:second_deriv_j}
  j''(u)\left[\delta u,\tau u\right] = (\alpha\,\delta u - p\,\delta y - y\,\delta p,\tau u)_{L^2(\Gamma)}.
 \end{equation}
 Next, we derive some stability and Lipschitz properties of $S$, $Z$, $S'$ and $Z'$.
 As the following results require different assumptions on $f$, $y_d$ and $g$ we simply
 assume the most restrictive ones, this is, 
 \begin{equation*}
   f, y_d\in L^\infty(\Omega),\qquad g\in H^{1/2}(\Gamma).
 \end{equation*}
 Moreover, we will hide the dependency on these quantities in the generic constant
 to simplify the notation.
 \begin{lemma}\label{lem:stability_SZ}
   Let $u\in L^2(\Gamma)$ satisfy the assumption \eqref{eq:ass_coercivity}.
   The control-to-state operator $S$ satisfies the following inequalities:
   \begin{align*}
     \|S(u)\|_{H^1(\Omega)} &\le c,  &&\\
     \|S(u)\|_{H^{3/2}(\Omega)} + \|S(u)\|_{H^1(\Gamma)} &\le c\,(1+\|u\|_{L^{p_1}(\Gamma)}), \\
     \|S(u)\|_{L^\infty(\Omega)} &\le c\,(1+\|u\|_{L^{p_2}(\Gamma)})^2,
   \end{align*}
   with $p_1 > 2$ and $p_2\ge 2$ for $n=2$, and $p_1\ge 4$ and $p_2 > 8/3$ for $n=3$.
   The estimates remain valid when replacing the operator $S$ by the control-to-adjoint operator $Z$.
 \end{lemma}
 \begin{proof}
   The inequalities for $S$ are a direct consequence of Lemmata
   \ref{lem:lax_milgram} and \ref{lem:props_S}. The inequalities for $Z$ can
   be derived with similar arguments, but the right-hand side of the adjoint
   equation involves the corresponding state $S(u)$. However, in all cases the
   norms of $S(u)-y_d$ can be bounded by $c\,(1+\|S(u)\|_{H^1(\Omega)})\le c$.
 \end{proof}
 
 \begin{lemma}\label{lem:stability_SZ_prime}
   Given are $u,\delta u\in L^2(\Gamma)$ and it is assumed that $u$ 
   satisfies \eqref{eq:ass_coercivity}.
   Then, the following stability estimates hold true:   
   \begin{align*}
     \|S'(u)\delta u\|_{H^1(\Omega)} &\le c\,\|\delta u\|_{L^2(\Gamma)},\\
     \|S'(u)\delta u\|_{H^{3/2}(\Omega)} &\le c\left(1+\|u\|_{L^p(\Gamma)}\right)^3 \|\delta u\|_{L^2(\Gamma)},
   \end{align*}
   with $p>2$ for $n=2$ and $p\ge 4$ for $n=3$.
   The estimates remain valid when replacing $S'$ by $Z'$.
 \end{lemma}
 \begin{proof}
   In the following we write $y:=S(u)$ and $\delta y = S'(u)\delta u$.
   The stability in $H^1(\Omega)$ follows directly from Lemma \ref{lem:lax_milgram}
   and the estimate
   \begin{equation}\label{eq:multiplication_H-12}
     \|\delta u\,y\|_{H^{-1/2}(\Gamma)} = \sup_{\genfrac{}{}{0pt}{}{\varphi\in H^{1/2}(\Gamma)}{\varphi\not\equiv 0}} \frac{\left(\delta u \,y,\varphi\right)_{L^2(\Gamma)}}{\|\varphi\|_{H^{1/2}(\Gamma)}}
       \le c\,\|\delta u\|_{L^2(\Gamma)}\,\|y\|_{H^1(\Omega)},
     \end{equation}
     which follows from the same arguments used already in \eqref{eq:remainder_term}.
     The boundedness of $y:=S(u)$ in $H^1(\Omega)$ can be found in the previous Lemma.
     The estimate in the $H^{3/2}(\Omega)$-norm follows analogously with Lemma \ref{lem:props_S}a)
     and
     \[\|y\,\delta u\|_{L^2(\Gamma)}\le c\,\|y\|_{L^\infty(\Omega)}\,\|\delta u\|_{L^2(\Gamma)}\]
     and the stability in $L^\infty(\Omega)$ proved in Lemma \ref{lem:stability_SZ}.

     The estimates for $Z'$ are deduced with similar techniques.
     With the a priori estimate from Lemma \ref{lem:props_S}a)
     and the embedding $H^1(\Omega)\hookrightarrow L^r(\Omega)$ which holds for $r <\infty$ ($n=2$)
     or $r\le 6$ ($n=3$) we get
   \begin{align*}
     \|Z'(u)\delta u\|_{H^{3/2}(\Omega)} &\le c \left(1+\|u\|_{L^p(\Gamma)}\right)
     \left(\|p\,\delta u\|_{L^2(\Gamma)} + \|\delta y\|_{H^1(\Omega)}\right) \\
     &\le c\left(1+\|u\|_{L^p(\Gamma)}\right)
     \left(1+\|p\|_{L^\infty(\Gamma)}\right)\|\delta u\|_{L^2(\Gamma)}
   \end{align*}
   with $p=Z(u)$.
   The stability of $Z$ in $L^\infty(\Omega)$ is discussed in the previous Lemma.
 \end{proof}
 
 \begin{lemma}\label{lem:lipschitz_SZ}
  Let $u,v\in L^2(\Gamma)$ satisfy assumption \eqref{eq:ass_coercivity}. 
  Then, the following Lipschitz-estimates hold:
  \begin{align*}
    \|S(u) - S(v)\|_{H^1(\Omega)} &\le c\,\|u-v\|_{L^2(\Gamma)}, \\
    \|S'(u)\delta u - S'(v)\delta u\|_{H^1(\Omega)}
                                  &\le c\,\|u-v\|_{L^2(\Gamma)}\|\delta u\|_{L^2(\Gamma)}.
  \end{align*}
  The estimates are also valid when replacing $S$ by $Z$ and $Z$ by $Z'$.
\end{lemma}
\begin{proof}
  The estimates for $S$ and $S'$ follow directly from Lemma
  \ref{lem:lipschitz_general} and the stability estimates for $S$ and $S'$ 
  in $H^1(\Omega)$ proved in the Lemmata \ref{lem:stability_SZ} and \ref{lem:stability_SZ_prime}.
  The Lipschitz estimate for $Z$ is proved in a similar way. In this case
  one has to apply the Lipschitz estimate shown for 
  $S$ to the term $\|S(u)-S(v)\|_{H^1(\Omega)}$ appearing due to the differences in
  the right-hand sides.  
  With the same idea we show the Lipschitz estimate for
  $Z'$. Using again Lemma~\ref{lem:lipschitz_general} we get
  \begin{align*}
    &\|Z'(u)\delta u - Z'(v)\delta u\|_{H^1(\Omega)}
    \le c\,\Big(\|u-v\|_{L^2(\Gamma)}\,\|Z'(u)\delta u\|_{H^1(\Omega)}  \\
    &\qquad+ \|S'(u)\delta u - S'(v)\delta v\|_{H^1(\Omega)}
      + \|\delta u\,(Z(u) - Z(v))\|_{H^{-1/2}(\Gamma)}\Big).
  \end{align*}
  It remains to bound the three terms on the right-hand side.
  To this end, we apply Lemma \ref{lem:stability_SZ_prime}
  to the first term, the Lipschitz estimate for $S'(\cdot)\delta u$ to the second term,
  and the multiplication rule \eqref{eq:multiplication_H-12} with $y=Z(u) - Z(v)$
  as well as the Lipschitz estimate for $Z$ to the third term. 
\end{proof}
As the optimal control problem is non-convex we have to deal with local solutions.
For some local solution $\bar u\in U_{ad}$ we require the following second-order
sufficient condition:
\begin{assumption}[SSC]\label{ass:ssc}
  The objective functional is locally convex near the local solution $\bar u$, i.\,e.,
  a constant $\delta > 0$ exists such that
  \begin{equation}\label{eq:ssc}
    j''(\bar u)(v,v) \ge \delta\,\|v\|_{L^2(\Gamma)}^2 \qquad\forall v\in L^2(\Gamma).
  \end{equation}
\end{assumption}
With standard arguments one can show that each function $\bar u\in U_{ad}$ fulfilling the
first-order necessary condition \eqref{eq:opt_cond} and the second-order sufficient condition
\eqref{eq:ssc} is indeed a local solution and satisfies the quadratic growth condition
\begin{equation*}
  j(\bar u) \le j(u) - \gamma\,\|u-\bar u\|_{L^2(\Gamma)}^2\qquad \forall u\in B_\tau(\bar u),
\end{equation*}
with certain constants $\gamma,\tau > 0$.
The author is aware that there are weaker assumptions which are sufficient for 
local minima, for instance one could formulate \eqref{eq:ssc} for all directions
$v$ from a critical cone. However, with this assumption the
convergence proof for the postprocessing approach presented in Section
\ref{sec:postprocessing} requires some more careful investigations, in
particular the construction of a modified interpolant onto $U_{ad}$.
One possible solution for this issue can be found in \cite{KP14}.

Later, we will require the following Lipschitz estimate for the Hessian of $j$.
\begin{lemma}\label{lem:lipschitz_2nd_deriv}
  Let $u,v \in L^2(\Gamma)$ fulfilling \eqref{eq:ass_coercivity} be given.
  Then, the Lipschitz-estimate 
  \begin{equation*}
    \left\vert j''(u)(\delta u,\delta u) - j''(v)(\delta u,\delta u)\right\vert \le
    c\,\|\delta u\|_{L^2(\Gamma)}^2\,\|u-v\|_{L^2(\Gamma)}.
    \end{equation*}
    is valid for all $\delta u\in L^2(\Gamma)$.
\end{lemma}
\begin{proof}
  To shorten the notation we write
  $y_u=S(u)$, $p_u = Z(u)$, $\delta y_u = S'(u)\delta u$ and $\delta p_u = Z'(u)\delta u$.
  From the representation \eqref{eq:second_deriv_j} we obtain
  \begin{align*}
    & \left\vert j''(v)(\delta u,\delta u) - j''(u)(\delta u,\delta
      u)\right\vert \\
    &\quad\le \left\vert(p_u\, \delta y_u - p_v\, \delta y_v + y_u\,\delta p_u - y_v\,\delta
    p_v,\delta u)_{L^2(\Gamma)}\right\vert.
  \end{align*}
  We estimate the right-hand side using the Cauchy-Schwarz inequality, 
  the embedding $H^1(\Omega)\hookrightarrow L^4(\Gamma)$ and the Lipschitz
  estimates from Lemma \ref{lem:lipschitz_SZ} as well as the a priori estimates
  from Lemmata \ref{lem:stability_SZ} and
  \ref{lem:stability_SZ_prime}. This implies
\begin{align*}
  &\quad \left\vert(p_u\,\delta y_u - p_v\,\delta y_v,\delta u)_{L^2(\Gamma)}\right\vert \\
  &\le c\left(\|p_u - p_v\|_{H^1(\Omega)}\,\|\delta y_u\|_{H^1(\Omega)} +
    \|\delta y_u - \delta y_v\|_{H^1(\Omega)}\,\|p_v\|_{H^1(\Omega)}
  \right)\|\delta u\|_{L^2(\Gamma)}\\
  &\le c\,\|u-v\|_{L^2(\Gamma)}\,\|\delta u\|_{L^2(\Gamma)}^2.
\end{align*}
 With similar arguments we deduce 
 \begin{align*}
   &\quad \left\vert\left(y_u\,\delta p_u - y_v\,\delta p_v,\delta u\right)_{L^2(\Gamma)}\right\vert \\
   &\le c\left(\|y_u - y_v\|_{H^1(\Omega)}\,\|\delta p_u\|_{H^1(\Omega)} +
     \|\delta p_u - \delta p_v\|_{H^1(\Omega)}\,
     \|y_v\|_{H^1(\Omega)}\right)\|\delta u\|_{L^2(\Gamma)} \\
   &\le c\,\|u-v\|_{L^2(\Gamma)}\,\|\delta u\|_{L^2(\Gamma)}^2,
 \end{align*}
 and conclude the assertion.
\end{proof}
\begin{corollary}\label{cor:coervice_neighb}
  Let $\bar u\in U_{ad}$ be a local solution of \eqref{eq:target} satisfying Assumption \ref{ass:ssc}.
  Then, some $\varepsilon>0$ exists such that the inequality
  \[j''(u)(\delta u,\delta u) \ge \frac{\delta}2 \|\delta u\|^2_{L^2(\Gamma)}\]
  is valid for all $\delta u\in L^2(\Gamma)$ and $u\in L^2(\Gamma)$ with $\|u-\bar u\|_{L^2(\Gamma)}\le \varepsilon$.
\end{corollary}
\begin{proof}
  The assertion follows immediately from the previous Lemma. For further
  details we refer to \cite[Lemma 2.23]{KV09}.
\end{proof}

In the next Lemma we will collect some basic regularity results for the solution
of \eqref{eq:target}.
\begin{lemma}\label{lem:regularity_general}
  Let $\Omega\subset\mathbb R^n$, $n\in\{2,3\}$, be a Lipschitz domain.
  Each local solution  $\bar u\in U_{ad}$ of \eqref{eq:target}
  and the corresponding states $\bar y=S(\bar u)$, $\bar p=Z(\bar u)$ satisfy
  \begin{equation*}
    \bar u\in H^1(\Gamma)\cap L^\infty(\Gamma),\qquad \bar y,\bar p\in H^{3/2}(\Omega)
    \cap H^1(\Gamma)\cap C(\overline\Omega).
  \end{equation*}
\end{lemma}
\begin{proof}
  All regularity result, except $\bar u\in H^1(\Gamma)$, follow directly from Lemma \ref{lem:props_S}.
  To show $\bar u\in H^1(\Gamma)$ we apply the product rule
  \begin{equation*}
    \|\bar y\,\bar p\|_{H^1(\Gamma)} \le c\left(\|\bar y\|_{H^1(\Gamma)} \,\|\bar p\|_{L^\infty(\Omega)} + \|\bar y\|_{L^\infty(\Omega)} \,\|\bar p\|_{H^1(\Gamma)}\right) \le c
  \end{equation*}
  and confirm $\bar y\,\bar p\in H^1(\Omega)$. The desired result then follows after an
  application of the Stampacchia-Lemma, see \cite[p. 50]{KinderlehrerStampacchia1980},
  to the projection formula \eqref{eq:proj_formula}.
\end{proof}
Under additional assumptions on the geometry of $\Omega$ we can show even higher regularity. This is needed for the postprocessing approach studied in Section \ref{sec:postprocessing} where we will show
almost quadratic convergence of the control approximations.
\begin{lemma}\label{lem:regularity_improved}
  Let  $\Omega\subset\mathbb R^2$ be a bounded domain with a 
  $C^{1,1}$-boundary $\Gamma$. Then, there holds
  \begin{equation*}
    \bar u\in W^{1,q}(\Gamma)\cap H^{2-1/q}(\tilde\Gamma),\qquad \bar y, \bar p \in W^{2,q}(\Omega),
  \end{equation*}
  for all $\tilde\Gamma\subset\subset \mathcal A$ or $\tilde\Gamma\subset\subset \mathcal I$,
  where  $\mathcal A:=\{x\in\Gamma\colon u(x)\in\{u_a,u_b\}\}$ and
  $\mathcal I:=\Gamma\setminus\mathcal A$
  denote the active and inactive set, respectively.
\end{lemma}
\begin{proof}
  With the projection  formula \eqref{eq:proj_formula}, $\bar y, \bar p\in H^1(\Gamma)$ proved in
  Lemma \ref{lem:regularity_general}
  and the multiplication rule \cite[Theorem 1.4.4.2]{Gri85} we obtain
  $\bar u\in H^{1/2}(\Gamma)$. 
  From Lemma \ref{lem:props_S}c) we then conclude
  $\bar y, \bar p\in H^2(\Omega)\hookrightarrow W^{1,q}(\Gamma)$ for all $q<(1,\infty)$
  and a further application of the multiplication rule yields
  $\bar y\,\bar p\in W^{1,q}(\Gamma)$.
  From \eqref{eq:proj_formula} we conclude the property $\bar u\in W^{1,q}(\Gamma)$.
  Furthermore, we confirm that $\bar u\,\bar y, \bar u\,\bar p\in W^{1-1/q,q}(\Gamma)$
  and a standard shift theorem for the Neumann problem, compare also the
  technique used in the proof of Lemma \ref{lem:props_S}a), results in
  $\bar y,\bar p\in W^{2,q}(\Omega)$. Repeating the arguments above, i.\,e., using the multiplication
  rule and the projection formula, we obtain
  $\bar u\in W^{2-1/q,q}(\tilde\Gamma)\hookrightarrow H^{2-1/q}(\tilde\Gamma)$.
\end{proof}

We chose the assumptions of the previous Lemma in such a way that
the regularity is only restricted due to the projection formula. Of course, when
the control bounds are never active we could further improve the regularity results.

\section{Finite element approximation of the state equation}\label{sec:fem}

This section is devoted to the finite element approximation of the variational problem
\eqref{eq:weak_form}. While the results from the previous sections are valid for
arbitrary Lipschitz domains (unless otherwise explicitly assumed), we have to assume more smoothness
of the boundary $\Gamma$ in order to establish our discretization results:

\begin{enumerate}[align=left]
\item[\customlabel{ass:domain1}{\textbf{(A1)}}]
  The domain $\Omega\subset\mathbb R^n$, $n\in\{2,3\}$, possesses a
  Lipschitz continuous boundary $\Gamma$ which is piecewise $C^1$. 
\end{enumerate}

This definition includes arbitrary (possibly non-convex) polygonal or
polyhedral domains. Indeed, the regularity of solutions is in this case also
restricted by corner and edge singularities. However, for the first convergence result
we require only $H^{3/2}(\Omega)\cap H^1(\Gamma)$-regularity of the solution.
Later, we want to investigate improved discretization techniques for which
more regularity is needed. Then, we will use a stronger assumption on the domain.

First, we introduce shape-regular triangulations $\{\mathcal T_h\}_{h>0}$
of $\Omega$ consisting of triangles ($n=2$) or tetrahedra ($n=3$).
The elements $T$  may have curved edges/faces
such that the property 
\begin{equation*}
  \overline\Omega = \bigcup_{T\in\mathcal T_h} \overline T
\end{equation*}
is valid for an arbitrary domain $\Omega$.
Moreover, we assume that the triangulations are feasible in the sense of
Ciarlet \cite{Cia91}.

The mesh parameter $h>0$ is the maximal element diameter
\begin{equation*}
 h = \max_{T\in\mathcal T_h} h_T,\quad h_T:=\text{diam}(T).
\end{equation*}
The family of meshes $\{\mathcal T_h\}_{h>0}$ is assumed to be quasi-uniform,
this means some $\kappa > 0$ independent of $h$ exists such that
each element $T\in\mathcal T_h$ contains a ball with radius $\rho_T$
satisfying the estimate $\frac{\rho_T}{h} \ge \kappa$.
Each triangulation $\mathcal T_h$ of $\Omega$ induces also a triangulation $\mathcal E_h$
of the boundary $\Gamma$

By $F_T\colon \hat T\to T$ we denote the transformations from the reference
triangle/tetra\-hedron $\hat T$ to the world element $T\in\mathcal T_h$.
The transformations $F_T$ may be non-affine for elements with curved
faces. Here, we consider transformations of the form
\begin{equation*}
  F_T = \tilde F_T + \Phi_T,
\end{equation*}
with some affine function $\tilde F_T(\hat x) = \tilde B_T\hat x + \tilde b_T$, $\tilde B_T\in\mathbb R^{n\times n}$, $\tilde b\in\mathbb R^n$, chosen in such a way that if $T$ is a curved boundary
element, $\tilde T=\tilde F_T(\hat T)$ is an $n$-simplex whose vertices coincide with the vertices
of $T$. The assumed shape-regularity implies $\|\tilde B_T\|\le c\,h_T$ and
$\|\tilde B_T^{-1}\|\le h_T^{-1}$, see \cite[Theorem 15.2]{Cia91}.

To guarantee the validity of interpolation error estimates we assume:
\begin{enumerate}[align=left]
  \raggedright  
\item[\customlabel{ass:domain2}{\textbf{(A2)}}]
  The triangulations $\mathcal T_h$ are regular of order $2$ in the sense of \cite{Ber89},
  this is, for all sufficiently small $h>0$ there holds
  \begin{equation}\label{eq:props_trafo}
    \sup_{\hat x\in\hat T} \|D \Phi_T(\hat x)\cdot \tilde B_T^{-1}\| \le c < 1,\qquad 
    \sup_{\hat x\in \hat T}\|D^2 \Phi_T(\hat x)\| \le c h^2,
  \end{equation}
  for all $T\in\mathcal T_h$.
\end{enumerate}
There are multiple strategies to construct the mappings $F_T$ satisfying these assumptions
and we refer the reader for instance to \cite{Ber89,Sco73,Zla73}. Therein, it is assumed
that $\Gamma$ is piecewise $C^3$, only in the second reference $C^4$ is required.

The trial and test space is defined by
\begin{equation*}
  V_h:= \{ v_h\in C(\overline\Omega) \colon v_h = \hat v_h\circ F_T^{-1},\ \hat v_h \in \mathcal P_1(\hat T)\ \mbox{for all}\ T\in\mathcal T_h\}.
\end{equation*}
Next, we introduce an interpolation operator which maps functions from $W^{1,1}(\Omega)$ onto $V_h$.
Therefore, we partly use the quasi-interpolant proposed by Bernardi \cite{Ber89}, but
use a modification for boundary nodes as in \cite{SZ90}, see also \cite{Ape99}. To each interior node
$x_i$, $i=1,\ldots,N^{in}$, of $\mathcal T_h$, we associate the patch of elements
$\sigma_i:= \cup\{\bar T\colon T\in\mathcal T_h, x_i\in \bar T\}$. For the boundary nodes $x_i$,
$i=N^{in}+1,\ldots,N$, we define
$\sigma_i:=\cup\{\bar E\colon E\in\mathcal E_h, x_i\in \bar E\}$. 
Instead of using nodal values as for the Lagrange interpolant, we use the nodal values
of some regularized function computed by an $L^2$-projection over $\sigma_i$. Therefore, denote by $F_i\colon \hat\sigma_i\to \sigma_i$
a continuous transformation from a reference patch $\hat \sigma_i$ having diameter
$O(1)$ to $\sigma_i$.
The interpolation operator $\Pi_h\colon W^{1,1}(\Omega)\to V_h$ is defined as follows.
To each node $x_i$, $i=1,\ldots,N$, we associate a first-order polynomial
$p_i=\hat p_i\circ F_i^{-1}$ defined by
\begin{equation*}
  \int_{\hat\sigma_i} (\hat p_i - \hat u)\,\hat q = 0\qquad\forall q\in\mathcal P_1(\hat\sigma_i),
\end{equation*}
where $\hat u$ is chosen such that $u = \hat u\circ F_i^{-1}$.
The interpolation operator is defined by
\begin{equation*}
  \Pi_h v(x) = \sum_{i=1}^N p_i(x_i)\,\varphi_i(x),
\end{equation*}
where $\{\varphi_i\}_{i=1,\ldots,N}$ is the nodal basis of $V_h$.
Note, that due to the modification for boundary nodes, this operator is only applicable to $W^{1,1}(\Omega)$-functions. The desired interpolation properties remain valid. In particular, there holds
\begin{equation}\label{eq:int_error}
  \|u-\Pi_h u\|_{H^m(\Omega)} \le c h^{\ell-m} \|u\|_{H^{\ell}(\Omega)},\quad m\le \ell\le 2,\ \ell\ge 1,
\end{equation}
see \cite[Theorem 4.1]{Ber89}, \cite[Theorem 4.1]{SZ90}.
Due to the special choice of the patches $\sigma_i$ for the boundary nodes we get similar
interpolation error estimates on the boundary, this is,
\begin{equation}\label{eq:int_error_boundary}
    \|u-\Pi_h u\|_{H^m(\Gamma)} \le c h^{\ell-m}\|u\|_{H^{\ell}(\Gamma)}, \quad m\le \ell\le 2.
\end{equation}
The proof follows from the same arguments as in \cite[Theorem 4.1]{SZ90}.

The finite element solutions of \eqref{eq:weak_form}
are characterized by the variational formulations
\begin{equation}\label{eq:fem}
  \text{Find } y_h\in V_h\colon\quad a_u(y_h,v_h) = F(v_h)\qquad\forall v_h\in V_h.
\end{equation}
As in the continuous case one can show that \eqref{eq:fem} possesses a unique solution for each $h>0$.

With the usual arguments we can derive an error estimate for the approximation error in the energy-norm. 
\begin{lemma}\label{lem:h1_l2_error}
  Assume that \ref{ass:domain1} and \ref{ass:domain2} are satisfied and that
  the solution $y$ of \eqref{eq:weak_form} belongs to $H^{s}(\Omega)$ with some
  $s\in [1,2]$. Then, there holds the error estimate
\begin{align}
  \|y-y_h\|_{H^1(\Omega)}&\le c\,h^{s-1}\,\|y\|_{H^{s}(\Omega)}.\label{eq:h1_error}
\end{align}
\end{lemma}
\begin{proof}
  The proof follows from the C\'ea-Lemma and the interpolation error estimates \eqref{eq:int_error}.
\end{proof}
Of particular interest are error estimates on the boundary. This is required in order to derive
error estimates for boundary control problems. To this end, we prove first a suboptimal result
which is valid for arbitrary Lipschitz domains $\Omega$.
\begin{lemma}\label{lem:fe_error_suboptimal}
  Let the assumptions \ref{ass:domain1} and \ref{ass:domain2} be satisfied.
  It is assumed that the solution $y$ of \eqref{eq:weak_form} belongs to $H^{3/2}(\Omega)$.
  Moreover, the parameter $u$
  fulfills \eqref{eq:ass_coercivity} and belongs to
  $L^p(\Gamma)$ with $p>2$ for $n=2$ and $p\ge 4$ for $n=3$.
  Then, the error estimate
  \begin{equation*}
    \|y - y_h\|_{L^2(\Gamma)}
    \le c\,h\left(1+\|u\|_{L^p(\Gamma)}\right)\|y\|_{H^{3/2}(\Omega)}
    \le c\,h\left(1+\|u\|_{L^p(\Gamma)}\right)^2
  \end{equation*}
  holds, for all $h>0$.
\end{lemma}
\begin{proof}
  We introduce the dual problem
  \begin{equation*}
    \mbox{Find $w\in H^1(\Omega)$}\colon\quad
    a_u(v,w) = (y - y_h,v)_{L^2(\Gamma)}\qquad \forall v\in H^1(\Omega)
  \end{equation*}
  and obtain with the typical arguments of the Aubin-Nitsche trick
  \begin{align*}
    \|y - y_h\|_{L^2(\Gamma)}^2
    \le c\,\|y - y_h\|_{H^1(\Omega)}\,\|w-\Pi_h w\|_{H^1(\Omega)}
    \le c\,h\,\|w\|_{H^{3/2}(\Omega)}\,\|y\|_{H^{3/2}(\Omega)}.
  \end{align*}
  The last step is an application of Lemma \ref{lem:h1_l2_error} and the interpolation error estimate
  \eqref{eq:int_error}.
  The regularity required for the dual solution $w$ can be deduced from Lemma~\ref{lem:props_S}
  with $f\equiv0$ and $g=y-y_h$. Taking into account the a priori estimate
  \[
    \|w\|_{H^{3/2}(\Omega)} \le c\left(1+\|u\|_{L^p(\Gamma)}\right)\|y-y_h\|_{L^2(\Gamma)}
  \] we conclude the assertion.  
\end{proof}
If the solution is more regular, we can also show a higher convergence rate.
In this case we will use the H\"older inequality and a trace theorem to obtain
$\|y-y_h\|_{L^2(\Gamma)} \le \|y-y_h\|_{L^\infty(\Omega)}$, and insert the following result.
\begin{theorem}\label{thm:fe_error_linfty}
  Consider a planar domain domain $\Omega\in\mathbb R^2$.
  Let $u\in H^{1/2}(\Gamma)$ with $u\ge 0$ a.\,e., and assume that  
  \ref{ass:domain1} and \ref{ass:domain2} are satisfied.
  Assume that the solution $y$ of \eqref{eq:weak_form} belongs to
  $y\in W^{2,q}(\Omega)$ with $q\in[2,\infty)$.
  Then, the error estimate
  \begin{equation*}
  \|y-y_h\|_{L^\infty(\Omega)} \le c\,h^{2-2/q}\,\lnh\,\|y\|_{W^{2,q}(\Omega)}
\end{equation*}
is valid.
\end{theorem}
The proof requires rather technical arguments and is postponed to the appendix.

\section{The discrete optimal control problem}\label{sec:discrete_optimal_control}

In the following we investigate the discretized optimal control problem:
\begin{equation}\label{eq:discrete_target}
\text{Find}\ u_h\in U_h^{ad}\colon\quad  J_h(y_h,u_h) := \frac12 \|y_h-y_d\|_{L^2(\Omega)}^2 + \frac\alpha2 \|u_h\|_{L^2(\Gamma)}^2 \to\min!
\end{equation}
subject to
\begin{equation*}
  y_h\in V_h,\quad a_{u_h}(y_h,v_h) = F(v_h)\qquad\forall v_h\in V_h.
\end{equation*}
The reduced objective functional is denoted by $j_h(u_h):=J_h(S_h(u_h),u_h)$.
We use piecewise linear finite elements to approximate the state $y$, i.\,e., the space
$V_h$ is defined as in the previous section. The controls are sought in the space
of piecewise constant functions,
\[U_h^{ad}:=\{w_h\in L^\infty(\Gamma)\colon w_h|_E \in\mathcal P_0\quad\forall E\in\mathcal E_h\}\cap U_{ad},\]
where $\mathcal E_h$ is the triangulation of the boundary induced by $\mathcal T_h$.

As in the continuous case we can derive a first-order necessary optimality condition which reads
\begin{align*}
  a_{u_h}(y_h,v_h) &= F(v_h) &&\mbox{for all}\ v_h\in V_h,\\
  a_{u_h}(v_h,p_h) &= (y_h-y_d,v_h)_{L^2(\Omega)} && \mbox{for all}\ v_h\in V_h,\\
  (\alpha\,u_h - y_h\, p_h,w_h-u_h)_{L^2(\Gamma)} &\ge 0 && \mbox{for all}\ w_h\in U_h^{ad}.
\end{align*}
The discrete control--to--state operator is denoted by $S_h\colon L^2(\Gamma)\to V_h$
and the control--to--adjoint operator by $Z_h\colon L^2(\Gamma)\to V_h$.
Analogous to the continuous case we compute the first and second derivatives of $j_h$ and obtain
\begin{equation}\label{eq:deriv_jh}
  j_h'(u)\delta u = \left(\alpha\,u - S_h(u)\,Z_h(u),\delta u\right)_{L^2(\Gamma)}
\end{equation}
and
\begin{equation}\label{eq:second_deriv_jh}
  j_h''(u)(\delta u,\tau u) = \left(\alpha\,\delta u - S_h(u)\, Z_h'(u)\delta u - S_h'(u)\delta u\, Z_h(u),\tau u\right)_{L^2(\Gamma)}.
\end{equation}
The first-order optimality condition reads in the short form
\begin{equation}\label{eq:opt_cond_discrete}
  (\alpha\,u_h - S_h(u_h)\,Z_h(u_h), w_h-u_h)_{L^2(\Gamma)} \ge 0 \qquad \mbox{for all}\ w_h\in U_h^{ad}.
\end{equation}

\subsection{Properties of the discrete control--to--state/adjoint operator}

In Section~\ref{sec:optimal_control} we have derived several stability and Lipschitz properties for
the operators $S$, $Z$, $S'$ and $Z'$. Here, we will derive the discrete analogues that are needed
in the following. Throughout this section we assume that \ref{ass:domain1} and \ref{ass:domain2}
are fulfilled.

\begin{lemma}\label{lem:stability_ShZh}
  There hold the following properties:
  \begin{align*}
    \|S_h(u)\|_{H^1(\Gamma)}  &\le c\left(1+\|u\|_{L^{p_1}(\Gamma)}\right)^2,\\
    \|S_h(u)\|_{L^\infty(\Omega)} &\le c\left(1+\|u\|_{L^{p_2}(\Gamma)}\right)^2,
  \end{align*}
  for $p_1, p_2 >2$ for $n=2$ and $p_1\ge 4$, $p_2> 4$ for $n=3$.
  These estimates remain valid when replacing $S_h$ by $Z_h$.
\end{lemma}
\begin{proof}
  We start with the estimate in the $H^1(\Gamma)$-norm.
  With the triangle inequality and an inverse estimate we obtain
  \begin{align*}
    \|S_h(u)\|_{H^1(\Gamma)}
    &\le c \left(\|S(u) - S_h (u)\|_{H^1(\Gamma)}
      + \|S(u)\|_{H^1(\Gamma)} \right)\\
    &\le c\,\big( \|S(u) - \Pi_h S(u)\|_{H^1(\Gamma)}
      + h^{-1}  \|S(u) - \Pi_h S(u)\|_{L^2(\Gamma)} \\
    &\quad + h^{-1}\,\|S(u) - S_h(u)\|_{L^2(\Gamma)} +
    \|S(u)\|_{H^1(\Gamma)}\big).
  \end{align*}
  The first two terms are bounded by the last one due to
  \eqref{eq:int_error_boundary} and it remains to apply the stability estimate from
  Lemma \ref{lem:stability_SZ}.  For the third term we apply the error estimate from Lemma
  \ref{lem:fe_error_suboptimal}. This implies the first estimate.

  We prove the maximum norm estimate only for the case $n=3$.
  In the following, we write $y_h := S_h(u)$.
  We introduce the function $\tilde y\in H^1(\Omega)$ solving the problem
  \begin{equation*}
    -\Delta \tilde y + \tilde y = f\ \mbox{in}\ \Omega,\qquad \partial_n \tilde
    y = g-u\,y_h\ \mbox{on}\ \Gamma.
  \end{equation*}
  Obviously, $y_h$ is the Neumann Ritz-projection of $\tilde y$, i.\,e.,
  \begin{equation*}
    a^{\text N}(y_h - \tilde y,v_h)=\int_\Omega\left(\nabla(y_h-\tilde y)\cdot\nabla
      v_h + (y_h-\tilde y)\,v_h \right)= 0\quad\mbox{for all}\ v_h\in V_h.
  \end{equation*}
  Let $x^*\in \bar T^*$ with $T^*\in\mathcal
  T_h$ be the point where $|y_h|$ attains its maximum. 
  With an inverse inequality and the H\"older inequality we get
  \begin{align}\label{eq:yh_linfty_start}
    \|y_h\|_{L^\infty(\Omega)} &= |y_h(x^*)| \le c\,|T^*|^{-1}\,\|y_h\|_{L^1(T^*)} \nonumber\\
    &\le c \left(|T^*|^{-1}\,\|\tilde y-y_h\|_{L^1(T^*)} + \|\tilde y\|_{L^\infty(T^*)}\right) \nonumber\\
    &= c\,(\delta^h,\tilde y-y_h)_{L^2(\Omega)} + c\,\|\tilde y\|_{L^\infty(\Omega)},
  \end{align}
  where $\delta^h$ is a regularized delta function defined by 
  $\delta^h(x) = |T^*|^{-1}\sgn(\tilde y(x)-y_h(x))$ if $x\in T^*$ and $\delta^h(x)=0$
  otherwise. The second term on the
  right-hand side can be treated with the arguments used already in 
  the proof of Lemma \ref{lem:props_S}b), namely
  \begin{align*}
    \|\tilde y\|_{L^\infty(\Omega)}
    &\le c\left(\|f\|_{L^r(\Omega)} + \|g\|_{L^s(\Gamma)} + \|u\,y_h\|_{L^{s}(\Gamma)} \right)
  \end{align*}
  with $r>3/2$ and $s=2+\varepsilon$ with $\varepsilon>0$ sufficiently small such that the following
  arguments remain valid.
  Furthermore, we estimate the last term with the H\"older inequality
  with $p_2=4\,(2+\varepsilon)/(2-\varepsilon)$ and $p'=4$ (note that $1/{p_2}+1/p'=1/s$)
  and the embedding $H^1(\Omega)\hookrightarrow L^4(\Gamma)$. This yields
  \begin{equation*}
    \|u\,y_h\|_{L^{s}(\Gamma)} \le c\,\|u\|_{L^{p_2}(\Gamma)}\,\|y_h\|_{L^{4}(\Gamma)}
    \le c\,\|u\|_{L^{p_2}(\Gamma)}\,\|y_h\|_{H^1(\Omega)}.
  \end{equation*}
  It remains to exploit stability of $S_h$ in the $H^1(\Omega)$-norm to conclude
  \begin{equation}\label{eq:linfty_est_y_tilde}
    \|\tilde y\|_{L^\infty(\Omega)} \le c\,(1+\|u\|_{L^{p_2}(\Gamma)}).
  \end{equation}
  
  The estimate for the first term on the right-hand side of \eqref{eq:yh_linfty_start} is based
  on the ideas from \cite[Section 3.6]{Win08}. First, we introduce a 
  regularized Green's function $g^h\in H^1(\Omega)$ solving the variational
  problem
  $a^{\text N}(z,g^h) = (\delta^h,z)_{L^2(\Omega)}$ for all $z\in H^1(\Omega)$.
  The Neumann Ritz-projection of $g^h$ is denoted by $g_h^h$.
  Using the Galerkin orthogonality we obtain
  \begin{align}\label{eq:est_delta_err}
    (\delta^h,\tilde y-y_h) & = a^{\text N}(\tilde y-y_h,g^h)
    = a^{\text N}(\tilde y-\Pi_h \tilde y,g^h-g_h^h)\nonumber\\
    &\le c\,h^{1/2}\,\|\tilde y\|_{H^{3/2}(\Omega)}\,\|g^h\|_{H^1(\Omega)},
  \end{align}
  where the last step follows form the stability of the Ritz projection
  and the interpolation error estimate \eqref{eq:int_error}.
  To bound the $H^1(\Omega)$-norm of $g^h$ we apply the ellipticity of
  $a^{\text N}$, the definition of $g^h$, the H\"older inequality and an embedding to
  arrive at
  \begin{align*}
    c\,\|g^h\|_{H^1(\Omega)}^2
    &\le a^{\text N}(g^h,g^h) =(\delta^h,g^h)_{L^2(\Omega)} \\
    &\le c\,\|\delta^h\|_{L^{6/5}(\Omega)}\,\|g^h\|_{L^6(\Omega)}
      \le c\,h^{-1/2}\,\|g^h\|_{H^1(\Omega)}.
  \end{align*}
  The last step follows from the property $\|\delta^h\|_{L^{6/5}(\Omega)}\le
  c\,|T^*|^{-1/6}\le c\,h^{-1/2}$ that can be confirmed with a simple
  computation.
  Insertion into \eqref{eq:est_delta_err} and taking into account \eqref{eq:yh_linfty_start} and
  \eqref{eq:linfty_est_y_tilde} yields the desired stability estimate.

  The estimates for $Z_h$ follow in a similar way.
  One just has to replace $f$ by $S_h(u)-y_d$ and the result follows from the estimates
  proved already for $S_h(u)$.
\end{proof}
\begin{lemma}\label{lem:Sh_lipschitz}
  Assume that $u,v\in L^2(\Gamma)$ satisfy the assumption
  \eqref{eq:ass_coercivity}.
  Then, the Lipschitz estimate
  \begin{equation*}
    \|S_h(u) - S_h(v)\|_{H^1(\Omega)} \le c\,\|u-v\|_{L^2(\Gamma)}
  \end{equation*}
  holds.
\end{lemma}
\begin{proof}
  The proof follows with the same arguments as in the continuous case, see Lemmata
  \ref{lem:lipschitz_general} and \ref{lem:lipschitz_SZ}.
\end{proof}

Next, we discuss some error estimates for the approximation of the control-to-state and control-to-adjoint operator.
While estimates for $S_h$ and $Z_h$ are a direct consequence of Lemma 
\ref{lem:fe_error_suboptimal}, the results for the linearized operators $S_h'$ and $Z_h'$
require some more effort as for instance
$S'(u)\delta u - S_h'(u)\delta u$ does not fulfill the Galerkin orthogonality.
\begin{lemma}\label{lem:fe_error_h1}
  For each $u\in U_{ad}$ and $\delta u\in L^2(\Gamma)$
  the error estimates
  \begin{align*}
    \|S(u) - S_h(u)\|_{H^1(\Omega)} &\le c\,h^{1/2}\,(1+\|u\|_{L^p(\Gamma)}),\\
    \|S'(u)\delta u - S_h'(u)\delta u \|_{H^1(\Omega)}&\le c\,h^{1/2}\,(1+\|u\|_{L^p(\Gamma)})^3\,\|\delta u\|_{L^2(\Gamma)}
  \end{align*}
  are valid for $p>2$ for $n=2$ and $p\ge 4$ for $n=3$.
  The results are also valid when replacing $S$ and $S_h$ by $Z$ and $Z_h$,
  as well as $S'$ and $S_h'$ by $Z'$ and $Z_h'$, respectively.
\end{lemma}
\begin{proof}
  The first estimate is just a combination of the Lemmata \ref{lem:h1_l2_error} and
  \ref{lem:stability_SZ}.
  To show the estimate for the linearized operators we introduce again 
  the abbreviations $y:=S(u)$, $y_h:=S_h(u)$, $\delta y := S'(u)\delta u$ and $\delta y_h:= S_h'(u)\delta u$.
  Moreover, define the auxiliary function $\delta \tilde y_h\in V_h$ as the solution of
  \begin{equation*}
    a_u(\delta \tilde y_h,v_h) = (y\,\delta u,v_h)_{L^2(\Gamma)}\qquad \forall v_h\in V_h.
  \end{equation*}
  This function fulfills the Galerkin orthogonality, i.\,e.,
  $a_u(\delta y-\delta \tilde y_h,v_h) = 0$ for all $v_h\in V_h$.
  Hence, we obtain with Lemma \ref{lem:h1_l2_error} and the Lipschitz-property from Lemma
  \ref{lem:lipschitz_general} (note that this Lemma is also valid for the discrete solutions)
  \begin{align*}
    \|\delta y - \delta y_h\|_{H^1(\Omega)}
    &\le c\left(\|\delta y - \delta \tilde y_h\|_{H^1(\Omega)}
      + \|\delta \tilde y_h - \delta y_h\|_{H^1(\Omega)} \right)\\
    &\le c\left(h^{1/2}\,\|\delta y\|_{H^{3/2}(\Omega)}
      + \|\delta u\,(y-y_h)\|_{H^{-1/2}(\Gamma)}\right).
  \end{align*}
  For the first term we simply insert the second estimate from Lemma \ref{lem:stability_SZ_prime}.
  The second term on the right-hand side is further estimated by means of
  \cite[Theorem 1.4.4.2]{Gri85} and a trace theorem which yield
  \begin{equation*}
    \|\delta u\,(y-y_h)\|_{H^{-1/2}(\Gamma)} \le c\,\|\delta u\|_{L^2(\Gamma)} \,\|y-y_h\|_{H^1(\Omega)},
  \end{equation*}
  and the assertion follows after an application of the estimate shown already for $S(u)-S_h(u)$.
  The estimates for $Z$ and $Z'$ follow with similar arguments.
\end{proof}

\subsection{Convergence of the fully discrete solutions}
\label{sec:full_discretization}

Throughout this subsection we assume that the properties \ref{ass:domain1} and \ref{ass:domain2}
are fulfilled. These assumptions are needed to guarantee the required regularity of the solution
and the validity of interpolation error estimates.

As the solutions of both the continuous and discrete optimal control problem \eqref{eq:target} and \eqref{eq:discrete_target}, respectively, are not unique
we have to construct a sequence of discrete local solutions converging towards a continuous one.
The first question which arises is whether such a sequence exists. To this end, we introduce
a localized problem
\begin{equation}\label{eq:discrete_local}
  j_h(u_h)\to\min! \quad \mbox{s.\,t.}\ u_h\in U_h^{ad}\cap B_\varepsilon(\bar u),
\end{equation}
where $\bar u\in U_{ad}$ is a fixed local solution of \eqref{eq:target} and $\varepsilon>0$ is some small parameter. First, we show that this problem possesses a unique local solution
which would immediately follow if we could show that the coercivity discussed in Corollary
\ref{cor:coervice_neighb} is transferred to the discrete case.
The following arguments are similar to the investigations in \cite{CMT05}, in
particular Theorem 4.4 and 4.5 therein.
\begin{lemma}\label{lem:discrete_coerc}
  Let $\bar u\in U_{ad}$ be a local solution of \eqref{eq:target}.
  Assume that $\varepsilon>0$ and $h>0$ are sufficiently small. Then, the inequality
  \begin{equation*}
    j_h''(u)\delta u^2\ge \frac\delta4\|\delta u\|_{L^2(\Gamma)}^2
  \end{equation*}
  is valid for all $u$ satisfying $\|u-\bar u\|_{L^2(\Gamma)} \le
  \varepsilon$. 
\end{lemma}
\begin{proof}
  With the explicit representations of $j''$ and $j_h''$ from \eqref{eq:second_deriv_j}
  and \eqref{eq:second_deriv_jh}, respectively, and Corollary \ref{cor:coervice_neighb}, we obtain
  \begin{align}\label{eq:discrete_corec_initial}
    &\qquad \frac\delta2 \|\delta u\|_{L^2(\Gamma)}^2
    \le \left(j''(u)\delta u^2 - j_h''(u)\delta u^2\right) +
      j_h''(u)\delta u^2 \nonumber\\
    &\le \Big( \|y_h\, \delta p_h - y\, \delta p\|_{L^2(\Gamma)}
     + \|\delta y_h\, p_h -  \delta y\,p\|_{L^2(\Gamma)} \Big)\,\|\delta u\|_{L^2(\Gamma)} + j_h''(u)\delta u^2,
  \end{align}
  with $y=S(u)$, $p = Z(u)$, $\delta y= S'(u)\delta u$
  and $\delta p=Z'(u)\delta u$, and the discrete analogues $y_h=S_h(u)$,
  $p_h = Z_h(u)$, $\delta y_h= S_h'(u)\delta u$ and $\delta p_h=Z_h'(u)\delta u$.
  It remains to bound the two norms in parentheses appropriately. Therefore, we apply
  the triangle inequality,  
  the stability properties for $S'$, $S_h$, $Z'$ and $Z_h$ from
  Lemmata \ref{lem:stability_SZ}, \ref{lem:stability_SZ_prime} and \ref{lem:stability_ShZh}
  as well as the error estimates from Lemma \ref{lem:fe_error_h1}.
  Note that the control bounds provide the regularity for $u$ that is required for these estimates.
  As a consequence we obtain
  \begin{align*}
    \|y_h\, \delta p_h - y\, \delta p\|_{L^2(\Gamma)}
    & \le c\left(\|y-y_h\|_{H^1(\Omega)}\,\|\delta p\|_{H^1(\Omega)}
      + \|\delta p-\delta p_h\|_{H^1(\Omega)}\,\|y_h\|_{H^1(\Omega)}\right) \\
    &\le c\,h^{1/2}\,\|\delta u\|_{L^2(\Gamma)}.
  \end{align*}
  With similar arguments we can show
  \begin{align*}
    \|\delta y_h\, p_h -  \delta y \,p\|_{L^2(\Gamma)}
    & \le c \left(\|\delta y - \delta y_h\|_{H^1(\Omega)}\,\|p_h\|_{H^1(\Omega)} 
      +  \|p - p_h\|_{H^1(\Omega)}\,\|\delta y\|_{H^1(\Omega)}\right) \\
    &\le c\, h^{1/2}\,\|\delta u\|_{L^2(\Gamma)}.
  \end{align*}
  The previous two estimates together with \eqref{eq:discrete_corec_initial} imply
  \begin{equation*}
    \frac\delta2\, \|\delta u\|_{L^2(\Gamma)}^2
    \le c\,h^{1/2}\,\|\delta u\|_{L^2(\Gamma)}^2 + j_h''(\bar u)\delta u^2.
  \end{equation*}
  Choosing $h$ sufficiently small such that $c\, h^{1/2} \le \frac\delta4$ leads to the assertion.  
\end{proof}

\begin{theorem}\label{thm:convergence_fully_discrete}
  Let $\bar u\in U_{ad}$ be a local solution of \eqref{eq:target}
  satisfying Assumption~\ref{ass:ssc}.  
  Assume that $\varepsilon>0$ and $h_0>0$ are sufficiently small. Then, the auxiliary problem \eqref{eq:discrete_local} possesses a unique solution for each $h\le h_0$ denoted by $\bar u_h^\varepsilon$, and there holds
  \begin{equation*}
    \lim_{h\to 0} \|\bar u-\bar u_h^\varepsilon\|_{L^2(\Gamma)} = 0.
  \end{equation*}
\end{theorem}
\begin{proof}
  The existence of at least one solution of \eqref{eq:discrete_local} follows immediately
  from the compactness and non-emptyness of  $U_h^{ad}\cap B_\varepsilon(\bar u)$.
  Note that $Q_h \bar u\in U_h^{ad}\cap B_{\varepsilon}(\bar u)$ for sufficiently small $h>0$, this
  confirms that the feasible set is non-empty.
  Due to Lemma \ref{lem:discrete_coerc} this solution is unique.

  Moreover, the family $\{\bar u_h^\varepsilon\}_{h\le h_0}$ is bounded and hence, a weakly convergent 
  sequence $\{\bar u_{h_k}^\varepsilon\}_{k\in\mathbb N}$ with $h_k \searrow 0$
  exists.
  The weak limit is denoted by $\tilde u\in L^2(\Gamma)$
  and from the convexity of the feasible set we deduce $\tilde u\in U_h^{ad}\cap B_\varepsilon(\bar u)$.
  Without loss of generality it is assumed that $\bar u_h^\varepsilon\rightharpoonup \tilde u$ in $L^2(\Gamma)$ as $h\searrow 0$.

  Next, we show that $\tilde u$ is a local minimum of the continuous problem.
  First, we show the convergence of the corresponding states
  which follows with the arguments from \cite{CM02}.
  First, we employ the triangle inequality to get
  \begin{equation}\label{eq:conv_state_triangle}
    \|S(\tilde u) - S_h(\bar u_h^\varepsilon)\|_{H^1(\Omega)} \le
    c\left( \|S(\tilde u) - S_h(\tilde u)\|_{H^1(\Omega)} + \|S_h(\tilde u) - S_h(\bar u_h^\varepsilon)\|_{H^1(\Omega)} \right).
  \end{equation}
  For the first term on the right-hand side
  we exploit convergence of the finite element method proved in
  Lemma \ref{lem:fe_error_h1}  which yields
  \begin{equation*}
    \|S(\tilde u) - S_h(\tilde u)\|_{H^1(\Omega)}\to 0,\quad h\searrow 0.
  \end{equation*}
  With similar arguments as in the proof of Lemma \ref{lem:lipschitz_general}
  we moreover deduce
  \begin{align*}
    \|S_h(\tilde u) - S_h(\bar u_h^\varepsilon)\|_{H^1(\Omega)}^2
    &= -(\tilde u\,S_h(\tilde u) - \bar u_h^\varepsilon\,S_h(\bar u_h^\varepsilon), S_h(\tilde u) - S_h(\bar u_h^\varepsilon))_{L^2(\Gamma)} \\
    &= -((\tilde u-\bar u_h^\varepsilon)\,S_h(\tilde u), S_h(\tilde u) - S_h(\bar u_h^\varepsilon))_{L^2(\Gamma)} \\
    &- \int_\Gamma \bar u_h^\varepsilon\,(S_h(\tilde u) - S_h(\bar u_h^\varepsilon))^2.
  \end{align*}
  The integral term on the right-hand side is non-negative due to the lower control bounds
  $\bar u_h^\varepsilon \ge u_a\ge 0$. We can bound the first term on the right-hand side with
  the Cauchy-Schwarz inequality and the multiplication rule from \cite[Theorem 1.4.4.2]{Gri85}
  which provides
  \begin{align*}
    &    \|S_h(\tilde u) - S_h(\bar u_h^\varepsilon)\|_{H^1(\Omega)}^2  \le \|\tilde u - \bar u_h^\varepsilon\|_{H^{-s}(\Gamma)}
      \, \|S_h(\tilde u)\|_{H^1(\Gamma)} \, \|S_h(\tilde u) - S_h(\bar u_h^\varepsilon)\|_{H^1(\Omega)}
  \end{align*}
  for arbitrary $s\in(0,1/2)$.
  Note that there holds $\|\tilde u - \bar u_h^\varepsilon\|_{H^{-s}(\Gamma)}\to 0$ for $h\searrow 0$
  due to the compact embedding $L^2(\Gamma)\hookrightarrow H^{-s}(\Gamma)$, $s>0$.
  It remains to bound the second factor on the right-hand side by an application of Lemma
  \ref{lem:stability_ShZh} and to divide the whole estimate by the third factor.
  After insertion of this estimate into \eqref{eq:conv_state_triangle}
  we obtain the strong convergence of the states, this is, 
  \begin{equation}\label{eq:strong_states_convergence}
    \|S(\tilde u) - S_h(\bar u_h^\varepsilon)\|_{H^1(\Omega)} \to 0\quad\mbox{for}\quad h\searrow 0.
  \end{equation}
  Next, we show that $\tilde u$ is a local solution of the continuous problem \eqref{eq:target}.
  To this end we exploit \eqref{eq:strong_states_convergence} and the lower semi-continuity of
  the norm map to arrive at
  \begin{equation}\label{eq:conv_target}
    j(\tilde u)
    \le \liminf_{h\searrow 0} j_h(\bar u_h^\varepsilon)
      \le \limsup_{h\searrow 0} j_h(\bar u_h^\varepsilon) 
    \le \limsup_{h\searrow 0} j_h(Q_h \bar u) \le j(\bar u).      
  \end{equation}
  The second to last step follows from the optimality of $\bar u_h^\varepsilon$ for
  \eqref{eq:discrete_local} and the admissibility of the $L^2(\Gamma)$-projection 
  $Q_h \bar u$ for sufficiently small $h>0$.
  The last step follows from the strong convergence of the $L^2(\Gamma)$-projection $Q_h$
  in $L^2(\Gamma)$. Note that this implies
  $\lim_{h\searrow 0} \|S_h(Q_h \bar u) - S(\bar u)\|_{L^2(\Omega)} = 0$.
  Due to Assumption \ref{ass:ssc} the solution $\bar u$ is unique within $B_\varepsilon(\bar u)$
  when $\varepsilon > 0$ is sufficiently small. This implies $\tilde u = \bar u$.
  Note that all ``$\le$'' signs in \eqref{eq:conv_target} then turn to ``$=$'' signs. 

  To conclude the strong convergence of the sequence $\{\bar u_h^\varepsilon\}_{h>0}$ 
  we show additionally the convergence of the norms.
  This follows from \eqref{eq:conv_target} and the strong convergence of the states
  from which we infer
  \begin{align*}
    \frac\alpha2\lim_{h\searrow 0} \|\bar u_h^\varepsilon\|_{L^2(\Gamma)}^2
    &= \lim_{h\searrow 0}\left(j_h(\bar u_h^\varepsilon) - \frac12\|S_h(\bar u_h^\varepsilon) - y_d\|_{L^2(\Omega)}^2\right) \\
    &= j(\bar u) - \frac12 \|S(\bar u) - y_d\|_{L^2(\Omega)}^2 = \frac\alpha2 \|\bar u\|_{L^2(\Gamma)}^2.
  \end{align*}
\end{proof}
The previous Lemma guarantees that each local solution $\bar u\in U_{ad}$
can be approximated by a sequence of local solutions of the discretized problems
\eqref{eq:discrete_local}. Due to $\bar u_h^\varepsilon\in B_\varepsilon(\bar u)$
and $\bar u_h^\varepsilon\rightarrow \bar u$ for $h\searrow 0$ (i.\,e., the constraint $\bar u_h^\varepsilon\in B_\varepsilon(\bar u)$ is never active),
the functions $\bar u_h^\varepsilon$ are local solutions of the
discrete problems \eqref{eq:discrete_target} provided that $h>0$ is small enough.
Hence, we neglect the superscript $\varepsilon$ in the following and denote by
$\bar u_h$ the sequence of discrete local solutions converging to the
local solution $\bar u$.

Next, we show linear convergence of the sequence $\bar u_h$.
\begin{theorem}\label{thm:convergence_full_disc}
  Let $\bar u\in U_{ad}$ be a local solution of \eqref{eq:target}
  which fulfills Assumption~\ref{ass:ssc}, and $\{\bar u_h\}_{h>0}$
  are local solutions of \eqref{eq:discrete_target} with $\bar u_h\to \bar u$ for $h\searrow 0$.
  Then, the error estimate
 \begin{equation*}
   \|\bar u-\bar u_h\|_{L^2(\Gamma)} \le \frac{c}{\sqrt \delta} h
 \end{equation*}
 holds.
\end{theorem}
\begin{proof}
  Let $\xi = \bar u+t(\bar u_h-\bar u)$ with $t\in(0,1)$.
  From Corollary \ref{cor:coervice_neighb} we obtain for sufficiently small $h$  the estimate
  \begin{align*}
    \frac{\delta}2 \|\bar u-\bar u_h\|_{L^2(\Gamma)}^2
    &\le j''(\xi)(\bar u-\bar u_h)^2\\
    &= j'(\bar u)(\bar u-\bar u_h) - j'(\bar u_h)(\bar u-\bar u_h),      
  \end{align*}
  where the last step follows from the mean value theorem for some $t\in (0,1)$.
  Next, we confirm with the first-order optimality conditions that
  \begin{equation*}
    j'(\bar u)(\bar u-\bar u_h) \le 0 \le j_h'(\bar u_h)(Q_h \bar u-\bar u_h)
  \end{equation*}
  with the $L^2(\Gamma)$ projection $Q_h$ onto $U_h$.
  Note that the property $Q_h\bar u\in U_{ad}$ is trivially satisfied.
  Insertion into the inequality above leads to
  \begin{equation}\label{eq:error_full_disc_start}
    \frac\delta2 \|\bar u-\bar u_h\|_{L^2(\Gamma)}^2 \le
    (j_h'(\bar u_h) - j'(\bar u_h))(Q_h \bar u-\bar u_h) - j'(\bar u_h)(\bar u-Q_h \bar u).
  \end{equation}
  An estimate for the second part follows from orthogonality of the $L^2(\Gamma)$-projection,
  this is,
  \begin{align}\label{eq:error_full_disc_1}
    j'(\bar u_h)(\bar u-Q_h \bar u) &= (\alpha \,\bar u_h + S(\bar u_h)\,Z(\bar u_h), \bar u-Q_h \bar u)_{L^2(\Gamma)} \nonumber\\
                       &= (S(\bar u_h)\,Z(\bar u_h) - Q_h(S(\bar u_h)\,Z(\bar u_h)), \bar u-Q_h \bar u)_{L^2(\Gamma)} \nonumber\\
                       &\le c\,h^2\,\|S(\bar u_h)\,Z(\bar u_h)\|_{H^1(\Gamma)}\,\|\bar u\|_{H^1(\Gamma)}.
  \end{align}
  Furthermore, we exploit the Leibniz rule and the stability properties for $S$ and $Z$ 
  from Lemma \ref{lem:stability_SZ} to obtain
  \begin{align}\label{eq:reg_Suh_Zuh}
    \|S(\bar u_h)\,Z(\bar u_h)\|_{H^1(\Gamma)}
    &\le c\Big(\|S(\bar u_h)\|_{H^1(\Gamma)} \,\|Z(\bar u_h)\|_{L^\infty(\Omega)}\nonumber\\
    &\phantom{\le} + \|S(\bar u_h)\|_{L^\infty(\Omega)}\, \|Z(\bar u_h)\|_{H^1(\Gamma)}\Big)
      \le c.    
  \end{align}
    
  Next, we discuss the first term on the right-hand side of \eqref{eq:error_full_disc_start}.
  Insertion of the definition of $j_h'$ and $j'$ and the stability of $Q_h$ yield
  \begin{align*}
    & (j_h'(\bar u_h) - j'(\bar u_h))(Q_h \bar u-\bar u_h)\\
    &\quad =(S_h(\bar u_h)\,Z_h(\bar u_h) - S(\bar u_h)\,Z(\bar u_h), Q_h(\bar u-\bar u_h))_{L^2(\Gamma)} \\
    &\quad \le c\,\Big(\|S_h(\bar u_h)\|_{L^\infty(\Gamma)}\, \|Z(\bar u_h) -
      Z_h(\bar u_h)\|_{L^2(\Gamma)} \nonumber\\
      &\qquad \phantom{\le} + \|Z(\bar u_h)\|_{L^\infty(\Gamma)}\, \|S(\bar u_h) - S_h(\bar u_h)\|_{L^2(\Gamma)}\Big)\, \|\bar u-\bar u_h\|_{L^2(\Gamma)} \\
    &\quad \le c\,h\left(\|S_h(\bar u_h)\|_{L^\infty(\Gamma)}\,\|Z(\bar
      u_h)\|_{H^{3/2}(\Omega)} + \|Z(\bar u_h)\|_{L^\infty(\Gamma)}\,\|S(\bar
      u_h)\|_{H^{3/2}(\Omega)}\right)\\
    &\qquad \times\|\bar u-\bar u_h\|_{L^2(\Gamma)}.
  \end{align*}
  In the last step we inserted the finite element error estimates from Lemma \ref{lem:fe_error_suboptimal}. Exploiting also the stability estimates from Lemmata \ref{lem:stability_SZ} and
  \ref{lem:stability_ShZh}
   we obtain
  \begin{equation*}
    (j_h'(\bar u_h) - j'(\bar u_h))(Q_h \bar u-\bar u_h)
    \le c\,h\,\|\bar u - \bar u_h\|_{L^2(\Gamma)}.
  \end{equation*}
  Together with \eqref{eq:error_full_disc_start}, \eqref{eq:error_full_disc_1}
  and \eqref{eq:reg_Suh_Zuh} we arrive at the assertion.
\end{proof} 

\subsection{Postprocessing approach}\label{sec:postprocessing}

In this section we consider the so-called postprocessing approach introduced in \cite{MR04}.
The basic idea is to compute an ``improved'' control $\tilde u_h$ by a pointwise evaluation of
the projection formula, i.\,e.,
\begin{equation}\label{eq:def_postprocessing}
  \tilde u_h := \Pi_{\text{ad}}\left(-\frac1\alpha [\bar y_h\,\bar p_h]_\Gamma\right),
\end{equation}
where $\bar y_h$ and $\bar p_h$ is the discrete state and adjoint state,
respectively, obtained by the full discretization approach discussed in
Section~\ref{sec:full_discretization}.
As we require higher regularity of the exact solution in order to observe a
higher convergence rate than for the full discretization approach,
we replace \ref{ass:domain1} by the stronger assumption
\begin{enumerate}[align=left]
\item[\customlabel{ass:domain1b}{\textbf{(A1')}}]
  The domain $\Omega$ is planar and its boundary is globally $C^3$.
\end{enumerate}
The most technical part of convergence proofs for this approach
is the proof of $L^2$-norm estimates for the state variables.
This is usually done by considering the following three terms separately:
\begin{align}\label{eq:postprocessing_basic}
  &\|\bar y-\bar y_h\|_{L^2(\Gamma)} \nonumber\\
  &\quad \le
  c \left(\|\bar y-S_h(\bar u)\|_{L^2(\Gamma)} + \|S_h(\bar u) - S_h(R_h \bar u)\|_{L^2(\Gamma)}
    + \|S_h(R_h \bar u) - \bar y_h\|_{L^2(\Gamma)}\right).
\end{align}
In \cite{MR04} $R_h\colon C(\Gamma)\to U_h$ is chosen as the midpoint interpolant.
We will construct and investigate such an operator in Appendix \ref{app:midpoint}.
Note that a definition of a midpoint interpolant on curved elements is not straight-forward.
The first term on the right-hand side of \eqref{eq:postprocessing_basic} is a finite element error in the $L^2(\Gamma)$-norm. We collect the required estimates in the following Lemma.
\begin{lemma}\label{lem:postprocessing_first}
  For all $q<\infty$ there hold the estimates
  \begin{align*}
    \|\bar y - S_h(\bar u)\|_{L^2(\Gamma)}&\le c\,h^{2-2/q}\,\lnh\, \|\bar y\|_{W^{2,q}(\Omega)} \\
    \|\bar p - Z_h(\bar u)\|_{L^2(\Gamma)}&\le c\,h^{2-2/q}\,\lnh\, \left(\|\bar p\|_{W^{2,q}(\Omega)} +
                                            \|\bar y\|_{H^2(\Omega)}\right).
  \end{align*}
\end{lemma}
\begin{proof}
  The first estimate follows from the H\"older inequality and the maximum norm estimate derived
  in Theorem \ref{thm:fe_error_linfty}. The second estimate requires an intermediate step.
  We denote by $p^h(\bar u)\in V_h$ the solution of the equation
  \begin{equation*}
    a_{\bar u}(p^h(\bar u),v_h) = (S(\bar u)-y_d,v_h)_{L^2(\Omega)}\qquad \forall v_h\in V_h.
  \end{equation*}
  As $p^h(\bar u)$ is the Ritz-projection of $\bar p$ we can apply Theorem \ref{thm:fe_error_linfty}
  again and obtain
  \begin{equation*}
    \|\bar p - p^h(\bar u)\|_{L^2(\Gamma)} \le c\,h^{2-2/q}\,\lnh\,\|\bar p\|_{W^{2,q}(\Gamma)}.
  \end{equation*}
  To show an estimate for the error between $p^h(\bar u)$ and $Z_h(\bar u)$ we
  test the equations defining both functions by $v_h = p^h(\bar u) - Z_h(\bar
  u)$, compare the proof of Lemma \ref{lem:lipschitz_general}.
  Together with the non-negativity of $\bar u$ we obtain
  \begin{align*}
    &\|p^h(\bar u)-Z_h(\bar u)\|_{H^1(\Omega)}^2\\
    &\quad = -\int_\Gamma \bar u\,(p^h(\bar u) - Z_h(\bar u))^2 
      + (S(\bar u) - S_h(\bar u),p^h(\bar u)-Z_h(\bar u))_{L^2(\Omega)} \\
    &\quad\le c\,h^2\,\|y\|_{H^2(\Omega)}\,\|p^h(\bar u)-Z_h(\bar u)\|_{H^1(\Omega)}.
  \end{align*}
  The last step follows from the estimate $\|S(\bar u) - S_h(\bar u)\|_{L^2(\Omega)} \le c\,h^2\,\|S(\bar u)\|_{H^2(\Omega)}$ which is a consequence of the Aubin-Nitsche trick.
  With the triangle inequality we conclude the desired estimate for the discrete
  control--to--adjoint operator  .
\end{proof}

To obtain an optimal error estimate for the second term we need an additional assumption
which is used in all contributions studying the postprocessing approach.
To this end, 
define the subsets $\mathcal K_2:=\cup\{\bar E\colon E\in\mathcal E_h,\ E\subset\mathcal A,\ \mbox{or}\ E\subset\mathcal I\}$ and $\mathcal K_1:=\Gamma\setminus \mathcal K_2$.
In the following we will assume that $\mathcal K_1$ satisfies 
\begin{equation}\label{eq:assumption}
  |\mathcal K_1|\le c\,h.
\end{equation}
The idea of this assumption is, that the control can only switch between active and inactive set on
$\mathcal K_1$. Only due to these switching points the regularity of the control is reduced,
see also Lemma \ref{lem:regularity_improved}.
One can in general expect that this happens at finitely many points
and thus, the assumption \eqref{eq:assumption} is not very restrictive.

As an intermediate result required to prove estimates for $Z_h(\bar u)-Z_h(R_h\bar u)$ in
$L^2(\Gamma)$, we need an estimate for $S_h(\bar u) - S_h(R_h\bar u)$ in $L^2(\Omega)$.
\begin{lemma}\label{lem:postprocessing_second_aux}
  For all $q<\infty$ there holds the estimate
  \begin{equation*}
    \|S_h(\bar u) - S_h(R_h \bar u)\|_{L^2(\Omega)} \le
    c\,h^{2-2/q}\left(1 + \|\bar u\|_{W^{1,q}(\Gamma)} + \|\bar u\|_{H^{2-1/q}(\mathcal K_2)}\right).
  \end{equation*}
\end{lemma}
\begin{proof}
  To shorten the notation we write $e_h:= S_h(\bar u) - S_h(R_h \bar u)$.
  Moreover, we introduce the function $w\in H^1(\Omega)$ solving the equation
  \begin{equation}\label{eq:dual_problem}
    a_{\bar u}(v,w) = (e_h,v)_{L^2(\Omega)}\qquad\forall v\in H^1(\Omega).
  \end{equation}
  This implies
  \begin{equation}\label{eq:superapprox_omega_begin}
    \|e_h\|_{L^2(\Omega)}^2 = a_{\bar u}(e_h,w-\Pi_h w) + a_{\bar u}(e_h,\Pi_h w).
  \end{equation}
  Next, we discuss both terms on the right-hand side separately.
  The first one is treated with the Cauchy-Schwarz inequality and
  the interpolation error estimate \eqref{eq:int_error}.
  These arguments lead to
  \begin{equation*}
    a_{\bar u}(e_h,w-\Pi_h w) \le c\,h\,\|e_h\|_{H^1(\Omega)}\,\|e_h\|_{L^2(\Omega)}.
  \end{equation*}
  The $H^1(\Omega)$-norm of $e_h$ is further estimated by the Lipschitz property from
  Lemma~\ref{lem:Sh_lipschitz} and an interpolation error estimate for the midpoint interpolant. This yields
  \begin{equation*}
    \|e_h\|_{H^1(\Omega)} \le c\,\|\bar u - R_h\bar u\|_{L^2(\Gamma)} \le c\,h\,\|\bar u\|_{W^{1,q}(\Gamma)}
  \end{equation*}
  for all $q\ge 2$.
  Insertion into the estimate above taking into account the stability estimates
  from Lemma \ref{lem:stability_ShZh} yields
  \begin{equation}\label{eq:superapprox_omega_first}
    a_{\bar u}(e_h,w-\Pi_h w)  \le c h^{2} \|\bar u\|_{W^{1,q}(\Gamma)}  \|e_h\|_{L^2(\Omega)}.
  \end{equation}
  
  Next, we consider the second term on the right-hand side of \eqref{eq:superapprox_omega_begin}.
  After a reformulation by means of the definition of $S_h$ we get
  \begin{align}\label{eq:superapprox_omega_second_0}
    a_{\bar u}(e_h,\Pi_h w)
    &= a_{\bar u} (S_h(\bar u),\Pi_h w) - a_{\bar u}(S_h(R_h\bar u), \Pi_h w) \nonumber\\
    &= a_{R_h\bar u} (S_h(R_h \bar u),\Pi_h w) - a_{\bar u}(S_h(R_h\bar u), \Pi_h w)\nonumber\\
    &= ((R_h\bar u - \bar u)\,S_h(R_h \bar u),\Pi_h w)_{L^2(\Gamma)}.
  \end{align}  
  We can further estimate this term with the interpolation error estimate
  from Lemma \ref{lem:int_error_Rh}
  \begin{align}\label{eq:superapprox_omega_second_1}
    &a_{\bar u}(e_h,\Pi_h w)
    \le c\,h^2\,\|\bar u\|_{H^1(\Gamma)}\,\|S_h(R_h\bar u)\,\Pi_hw\|_{H^1(\Gamma)} \nonumber\\
    &\qquad + \|S_h(R_h\bar u)\|_{L^\infty(\Gamma)}\,\|\Pi_h
      w\|_{L^\infty(\Gamma)}\,\sum_{E\in\mathcal E_h} |\int_E(\bar u-R_h\bar
      u)| \nonumber\\
    &\le c\,\Big(h^2 + \sum_{E\in\mathcal E_h} |\int_E(\bar u-R_h\bar
      u)|\Big)\left(1+\|\bar u\|_{H^1(\Gamma)}\right)
      \|S_h(R_h\bar u)\|_{H^1(\Gamma)}\, \|\Pi_h w\|_{H^1(\Gamma)}.
  \end{align}
  The last step follows from the embedding $H^1(\Gamma)\hookrightarrow
  L^\infty(\Gamma)$ and the multiplication rule
$\|u\,v\|_{H^1(\Gamma)} \le c\,\|u\|_{H^1(\Omega)}\,\|v\|_{H^1(\Gamma)}$,
see \cite[Theorem 1.4.4.2]{Gri85}. Both properties are only fulfilled in case of $n=2$.

  Let us discuss the terms on the right-hand side separately.
  For elements $E\subset\mathcal K_1$ we can exploit the assumption~\eqref{eq:assumption}
  which provides the estimate $\sum_{E\subset\mathcal K_1} |E| \le c\,h$ and the second
  interpolation error
  estimate from Lemma~\ref{lem:local_est_midpoint} to arrive at
  \begin{equation*}
    \sum_{\genfrac{}{}{0pt}{}{E\in\mathcal E_h}{E\subset\mathcal K_1}} |\int_E(\bar u-R_h\bar u)|
    \le c\,h\,\|\bar u-R_h\bar u\|_{L^\infty(\Gamma)}
    \le c\,h^{2-1/q}\,\|\nabla \bar u\|_{L^q(\Gamma)}.
  \end{equation*}
  Moreover, with the first estimate from Lemma~\ref{lem:local_est_midpoint} and the discrete
  Cauchy-Schwarz inequality we obtain for elements
  in $\mathcal K_2$ the estimate
  \begin{align*}
    \sum_{\genfrac{}{}{0pt}{}{E\in\mathcal E_h}{E\subset\mathcal K_2}} |\int_E(\bar u-R_h\bar u)|
    &\le c\,h^{5/2-1/q}\,\|\bar u\|_{H^{2-1/q}(\mathcal K_2)}\,
      \Big(\sum_{\genfrac{}{}{0pt}{}{E\in\mathcal E_h}{E\subset\mathcal K_2}}1\Big)^{1/2} \\
    &\le c\,h^{2-1/q}\,\|\bar u\|_{H^{2-1/q}(\mathcal K_2)}.
  \end{align*}
  The remaining terms on the right-hand side of
  \eqref{eq:superapprox_omega_second_1}
  can be treated with stability estimates for $S_h$ (see
  Lemma~\ref{lem:stability_ShZh}) and $R_h$,
  the estimate $\|\Pi_h w\|_{H^1(\Gamma)}\le c\,\|w\|_{H^1(\Gamma)}$
  stated in \eqref{eq:int_error_boundary} and the a~priori estimate
  $\|w\|_{H^1(\Gamma)} \le c\,\|e_h\|_{L^2(\Omega)}$ from Lemma \ref{lem:props_S}a).
  Insertion of the previous estimates into \eqref{eq:superapprox_omega_second_1}
  yields
  \begin{equation}\label{eq:superapprox_omega_second}
    a_{\bar u}(e_h,\Pi_h w) \le c\,h^{2-1/q}\left(1+\|\bar u\|_{W^{1,q}(\Gamma)} + \|\bar u\|_{H^{2-1/q}(\mathcal K_2)}\right)\|e_h\|_{L^2(\Omega)}.
  \end{equation}
  Note that we hide the of lower-order norms of $\bar u$ in the generic constant
  as these quantities may be estimated by means of the control bounds $u_a$ and $u_b$.
  Insertion of \eqref{eq:superapprox_omega_first} and \eqref{eq:superapprox_omega_second}
  into \eqref{eq:superapprox_omega_begin} and dividing by $\|e_h\|_{L^2(\Omega)}$
  implies the assertion.
\end{proof}

\begin{lemma}\label{lem:postprocessing_second}
    Under the assumption \eqref{eq:assumption} the estimates 
    \begin{align*}
      \|S_h(\bar u) - S_h(R_h \bar u)\|_{L^2(\Gamma)} &\le c\,h^{2-1/q}\, \left(1 + \|\bar u\|_{H^{2-1/q}(\mathcal K_2)} + \|\bar u\|_{W^{1,q}(\Gamma)}\right), \\
      \|Z_h(\bar u) - Z_h(R_h \bar u)\|_{L^2(\Gamma)} &\le c\,h^{2-1/q}\, \left(1 + \|\bar u\|_{H^{2-1/q}(\mathcal K_2)} + \|\bar u\|_{W^{1,q}(\Gamma)}\right)
    \end{align*}
    are valid for arbitrary $q\in[2,\infty)$.
\end{lemma}
\begin{proof}
  We will only prove the second estimate as the first one follows from the same technique
  and is even easier as the right-hand sides of the equations defining $S_h(\bar u)$ and $S_h(R_h\bar u)$ coincide. This is not the case for the control--to--adjoint operator.
  
  To shorten the notation we write $e_h:=Z_h(\bar u) - Z_h(R_h \bar u)$.
  As in the previous Lemma we rewrite the error by a duality argument using a dual problem
  similar to \eqref{eq:dual_problem} with solution $w\in H^1(\Omega)$, more precisely,
  \begin{equation*}
    a_{\bar u}(v,w) = (e_h,v)_{L^2(\Gamma)}\qquad \forall v\in H^1(\Omega).
  \end{equation*}
  This yields
  \begin{equation}\label{eq:superapprox_begin}
    \|e_h\|_{L^2(\Gamma)}^2 = a_{\bar u}(e_h,w-\Pi_h w) + a_{\bar u}(e_h, \Pi_h w).
  \end{equation}
  We rewrite the second expression in \eqref{eq:superapprox_begin} and get
  analogous to \eqref{eq:superapprox_omega_second_0}
  \begin{align}\label{eq:superapprox_first}
    &a_{\bar u}(e_h,\Pi_h w)\nonumber\\
    &\quad= a_{\bar u}(Z_h(\bar u), \Pi_h w) \pm a_{R_h\bar u}(Z_h(R_h \bar u), \Pi_h w) 
      - a_{\bar u}(Z_h(R_h \bar u), \Pi_h w)\nonumber\\
    &\quad = (S_h(\bar u) - S_h(R_h\bar u), \Pi_h w)_{L^2(\Omega)} 
      + ((R_h \bar u - \bar u)\, Z_h(R_h \bar u), \Pi_h w)_{L^2(\Gamma)}.
  \end{align}
  Note that the first term would not appear when deriving estimates for $S_h$ instead of $Z_h$ as 
  the equations defining $S_h(\bar u)$ and $S_h(R_h \bar u)$ have the same right-hand side.

  The first term can be treated with the Cauchy-Schwarz inequality,
  Lemma~\ref{lem:postprocessing_second_aux} and the estimate
  $\|\Pi_h w\|_{L^2(\Omega)} \le c\,\|w\|_{H^1(\Omega)}\le
  c\,\|e_h\|_{L^2(\Gamma)}$ which can be deduced from
  \eqref{eq:int_error} and Lemma~\ref{lem:lax_milgram} with $g=e_h$.
  These ideas lead to
  \begin{align}\label{eq:superapprox_Zhw1}
    &(S_h(\bar u) - S_h(R_h\bar u), \Pi_h w)_{L^2(\Omega)} \nonumber\\
    &\qquad \le c\,h^{2-1/q}\, \left(1 + \|\bar u\|_{W^{1,q}(\Gamma)} + \|\bar u\|_{H^{2-1/q}(\mathcal K_2)}\right)\,\|e_h\|_{L^2(\Gamma)}.
  \end{align}
  For the second term on the right-hand side of \eqref{eq:superapprox_first}
  we apply the same steps as for \eqref{eq:superapprox_omega_second}
  with the only modification that the a priori estimate
  $\|w\|_{H^1(\Gamma)} \le c \|e_h\|_{L^2(\Gamma)}$ from
  Lemma~\ref{lem:props_S}a) has to be employed. From this we infer
   \begin{align}\label{eq:superapprox_decomp}
     & ((R_h \bar u - \bar u)\,Z_h(R_h \bar u), \Pi_h w)_{L^2(\Gamma)}
       \nonumber\\
     &\quad \le c\, h^{2-1/q}\left(1+\|\bar u\|_{W^{1,q}(\Gamma)}
       + \|\bar u\|_{H^{2-1/q}(\mathcal K_2)}\right)
       \|Z_h(R_h \bar u)\|_{H^1(\Gamma)}\,\|\Pi_h w\|_{H^1(\Gamma)}\nonumber\\
     &\quad \le c\, h^{2-1/q}\left(1+\|\bar u\|_{W^{1,q}(\Gamma)}
       + \|\bar u\|_{H^{2-1/q}(\mathcal K_2)}\right)
       \|e_h\|_{L^2(\Gamma)}.
   \end{align}
   In the last step we used the boundedness of $Z_h(R_h \bar u)$, see Lemma \ref{lem:stability_ShZh}.
   Insertion of \eqref{eq:superapprox_Zhw1} and \eqref{eq:superapprox_decomp}
   into \eqref{eq:superapprox_first} 
   leads to
   \begin{equation}\label{eq:superapprox_Zhw}
     a_{\bar u}(e_h,\Pi_h w) \le c\,h^{2-1/q}\left(1 + \|\bar u\|_{W^{1,q}(\Gamma)}+ \|\bar u\|_{H^{2-1/q}(\mathcal K_2)}\right) \|e_h\|_{L^2(\Gamma)}.
   \end{equation}
   It remains to discuss the first term on the right-hand side of \eqref{eq:superapprox_begin}.
   We obtain with the boundedness of $a_{\bar u}$, the interpolation error estimate
   \eqref{eq:int_error} and Lemma~\ref{lem:props_S}a)
   \begin{equation}\label{eq:superapprox_third}
     a_{\bar u}(e_h,w-\Pi_h w) \le c\,h^{1/2}\,\|e_h\|_{H^1(\Omega)}\, \|w\|_{H^{3/2}(\Omega)}
     \le\,c\,h^{1/2}\,\|e_h\|_{H^1(\Omega)}\,\|e_h\|_{L^2(\Gamma)}.
   \end{equation}
   An estimate for the expression $\|e_h\|_{H^1(\Omega)}$ follows from the equality
   \begin{equation*}
     \|e_h\|_{H^1(\Omega)}^2 + (\bar u\,Z_h(\bar u) - R_h \bar u\, Z_h(R_h \bar u), e_h)_{L^2(\Gamma)} = (S_h(\bar u) - S_h(R_h\bar u), e_h)_{L^2(\Omega)}
   \end{equation*}
   which can be deduced by subtracting the equations for $Z_h(\bar u)$ and
   $Z_h(R_h \bar u)$ from each other.
   Rearranging the terms yields
   \begin{align*}
     \|e_h\|_{H^1(\Omega)}^2 &\le ((R_h \bar u - \bar u)\,Z_h(R_h \bar u),
                               e_h)_{L^2(\Gamma)} \\
     &- (\bar u \,e_h,e_h)_{L^2(\Gamma)} + \|S_h(\bar u) - S_h(R_h\bar u)\|_{L^2(\Omega)} \|e_h\|_{L^2(\Omega)}.
   \end{align*}
   The second term on the right-hand side can be bounded by zero as $\bar u\ge 0$.
   An estimate for the last term is proved in Lemma \ref{lem:postprocessing_second_aux}.
   For the first term we apply the estimate \eqref{eq:superapprox_decomp} with
   $\Pi_h w$ replaced by $e_h$. All together, we obtain
   \begin{equation}\label{eq:error_eh_H1}
     \|e_h\|_{H^1(\Omega)}^2 \le c\, h^{2-1/q}\left(1 + \|\bar
       u\|_{W^{1,q}(\Gamma)} + \|\bar u\|_{H^{2-1/q}(\mathcal K_2)}\right)
     \left(\|e_h\|_{H^1(\Gamma)} + \|e_h\|_{H^1(\Omega)}\right).
   \end{equation}
   Moreover, with an inverse inequality and a trace theorem we get
   \begin{equation*}
     \|e_h\|_{H^1(\Gamma)} \le c\,h^{-1/2}\,\|e_h\|_{H^{1/2}(\Gamma)} \le c\,h^{-1/2}\,\|e_h\|_{H^1(\Omega)}.
   \end{equation*}   
   Consequently, we deduce from \eqref{eq:error_eh_H1}
   \begin{equation*}
     \|e_h\|_{H^1(\Omega)} \le c\,h^{3/2-1/q}\left(1 + \|\bar
       u\|_{W^{1,q}(\Gamma)} + \|\bar u\|_{H^{2-1/q}(\mathcal K_2)}\right).
   \end{equation*}   
   Insertion into \eqref{eq:superapprox_third}  leads to
   \begin{equation*}
     a_{\bar u}(e_h,w-\Pi_h w) \le c\,h^{2-1/q}\,\left(1 + \|\bar
       u\|_{W^{1,q}(\Gamma)} + \|\bar u\|_{H^{2-1/q}(\mathcal K_2)}\right)\|e_h\|_{L^2(\Gamma)}.
   \end{equation*}
   Together with \eqref{eq:superapprox_Zhw} and \eqref{eq:superapprox_begin} we conclude the
   desired estimate for $Z_h$.
 \end{proof}

 \begin{lemma}\label{lem:postprocessing_third}
   Under the assumption \eqref{eq:assumption} there holds the estimate
   \begin{equation*}
     \|R_h \bar u - \bar u_h\|_{L^2(\Gamma)} \le c\,h^{2-2/q}\,\lnh
   \end{equation*}
   with \[c = c\left(\|\bar u\|_{H^{2-1/q}(\mathcal K_2)},\|\bar u\|_{W^{1,q}(\Gamma)},
       \|\bar y\|_{W^{2,q}(\Omega)},\|\bar p\|_{W^{2,q}(\Omega)}\right).\]
 \end{lemma}
 \begin{proof}
   We observe that each function $\xi:=t\, R_h \bar u + (1-t)\, \bar u_h$ for 
   $t\in [0,1]$ satisfies 
   \begin{equation*}
     \|\bar u - \xi\|_{L^2(\Gamma)} \le t \|\bar u - R_h\bar u\|_{L^2(\Gamma)} + (1-t)\|\bar u - \bar u_h\|_{L^2(\Gamma)} < \varepsilon,
   \end{equation*}
   for arbitrary $\varepsilon>0$ provided that $h$ is sufficiently small.
   This follows from the convergence of the midpoint interpolant, see Lemma \ref{lem:local_est_midpoint}, and
   convergence of $\bar u_h$ towards $\bar u$, see Theorem \ref{thm:convergence_fully_discrete}.
   Hence, with the coercivity of $j_h''$
   proved in Lemma \ref{lem:discrete_coerc} and the mean value theorem we conclude
   \begin{align*}
     \frac\delta4\,\|R_h \bar u - \bar u_h\|_{L^2(\Gamma)}^2
     &\le j_h''(\xi)(R_h \bar u - \bar u_h)^2 \\
     &= j_h'(R_h \bar u)(R_h \bar u - \bar u_h) - j_h'(\bar u_h)(R_h \bar u - \bar u_h).     
   \end{align*}
   For the latter term we exploit the discrete optimality condition and the fact that
   the continuous optimality condition holds even pointwise. This implies the inequality
   \begin{equation*}
     j_h'(\bar u_h)(R_h\bar u - \bar u_h) \ge 0 \ge (\alpha\,R_h \bar u -
     R_h(\bar y\,\bar p),R_h \bar u - \bar u_h)_{L^2(\Gamma)}.
   \end{equation*}
   Insertion into the estimate above implies
   \begin{align}\label{eq:supercloseness_begin}
     &\frac\delta4\|R_h \bar u - \bar u_h\|_{L^2(\Gamma)}^2 \nonumber\\
     &\quad \le \left(R_h(\bar y\,\bar p)-\bar y\,\bar p
       +\bar y\,\bar p -S_h(R_h \bar u)\,Z_h(R_h \bar u), R_h \bar u - \bar u_h\right)_{L^2(\Gamma)}.
   \end{align}
   The right-hand side can be decomposed into two parts.
   With appropriate intermediate functions we obtain for the latter one
   \begin{align*}
     &(\bar y\, \bar p - S_h(R_h \bar u)\,Z_h(R_h \bar u), R_h \bar u - \bar u_h)_{L^2(\Gamma)} \\
     &\quad = ((\bar y-S_h(R_h\bar u))\, \bar p + S_h(R_h\bar u)\,(\bar p-Z_h(R_h\bar u)), R_h\bar u-\bar u_h)_{L^2(\Gamma)} \\
     &\quad \le c\,\big(\|\bar y-S_h(R_h\bar u)\|_{L^2(\Gamma)}\,\|\bar p\|_{L^\infty(\Gamma)} \\
       &\qquad + \|\bar p-Z_h(R_h\bar u)\|_{L^2(\Gamma)}\, \|S_h(R_h\bar u)\|_{L^\infty(\Gamma)}
       \big)\,\|R_h\bar u-\bar u_h\|_{L^2(\Gamma)}.
   \end{align*}
   Moreover, we apply the triangle inequality and the estimates from
   Lemmata \ref{lem:postprocessing_first} and \ref{lem:postprocessing_second} 
   to deduce 
   \begin{align*}
     \|\bar y-S_h(R_h\bar u)\|_{L^2(\Gamma)}
     &\le \|\bar y-S_h(\bar u)\|_{L^2(\Gamma)} + \|S_h(\bar u)-S_h(R_h\bar u) \|_{L^2(\Gamma)} \\
     &\le c\,h^{2-2/q}\,\lnh.
   \end{align*}
   Analogously, one can derive an estimate for the term $\|\bar p - Z_h(R_h\bar u)\|_{L^2(\Gamma)}$.
   Moreover, we apply Lemmata \ref{lem:stability_SZ} and \ref{lem:stability_ShZh} to bound the norms
   of $p=Z(\bar u)$ and $S_h(R_h \bar u)$, respectively.
   All together we obtain the estimate
   \begin{align}\label{eq:supercloseness_first}
     &(\bar y\,\bar p-S_h(R_h \bar u)\,Z_h(R_h \bar u), R_h \bar u - \bar
       u_h)_\Gamma 
       \le c\,h^{2-2/q}\,\lnh\,\|R_h\bar u-\bar u_h\|_{L^2(\Gamma)}.
   \end{align}
   Next we discuss that part of \eqref{eq:supercloseness_begin} which involves 
   the term $R_h(\bar y\,\bar p)-\bar y\,\bar p$ in the first argument.
   With an application of the local estimate from Lemma \ref{lem:int_error_Rh}
   we obtain
   \begin{align}\label{eq:supercloseness_second}
     (R_h (\bar y\,\bar p) - \bar y\,\bar p), R_h \bar u - \bar u_h)_{L^2(\Gamma)}
     & = \sum_{E\in\mathcal E_h} [R_h\bar u - \bar u_h]_E\int_E
       \left(\bar y\,\bar p-R_h (\bar y\,\bar p))\right) \nonumber\\
     &\le c\,h^{2-1/q}\,\|R_h\bar u- \bar u_h\|_{L^2(\Gamma)}\,\|\bar y\,\bar p\|_{H^{2-1/q}(\Gamma)}.
   \end{align}
   With \cite[Theorem 1.4.4.2]{Gri85} and a trace theorem we conclude
   \begin{align*}
     \|\bar y\,\bar p\|_{H^{2-1/q}(\Gamma)}
     \le c\,\|\bar y\|_{W^{2-1/q,q}(\Gamma)}\,\|\bar p\|_{W^{2-1/q,q}(\Gamma)}
     \le c\,\|\bar y\|_{W^{2,q}(\Omega)}\,\|\bar p\|_{W^{2,q}(\Omega)}.
   \end{align*}
   Insertion of the estimates \eqref{eq:supercloseness_first} and \eqref{eq:supercloseness_second}
   into \eqref{eq:supercloseness_begin}, and dividing the resulting estimate by
   $\|R_h u - u_h\|_{L^2(\Gamma)}$, leads to the desired result.
 \end{proof}
 Now we are in the position to state the main result of this section.
 \begin{theorem}\label{thm:postprocessing}
   Let $(\bar y,\bar u,\bar p)$ be a local solution of \eqref{eq:opt_cond}
   satisfying Assumption \ref{eq:assumption}. Moreover, let $\{\bar u_h\}_{h>0}$ be a sequence of
   local solutions of \eqref{eq:opt_cond_discrete} such that for sufficiently small $\varepsilon,h_0>0$
   the property 
   \[\|\bar u - \bar u_h\|_{L^2(\Gamma)} < \varepsilon\qquad\forall h < h_0\]
   holds. Then, the error estimate
   \begin{equation*}
     \|\bar u - \tilde u_h\|_{L^2(\Gamma)} \le c\,h^{2-2/q}\,\lnh
   \end{equation*}
   is satisfied, where
   $c=c(\|\bar u\|_{W^{1,q}(\Gamma)},\|\bar u\|_{H^{2-1/q}(\mathcal K_2)},\|\bar y\|_{W^{2,q}(\Omega)}, \|\bar p\|_{W^{2,q}(\Omega)})$.
 \end{theorem}
 \begin{proof}
   With the projection formulas \eqref{eq:proj_formula} and \eqref{eq:def_postprocessing}, respectively, the non-expansivity of the operator $\Pi_{ad}$ and the triangle inequality we obtain
   \begin{align*}
     \|\bar u - \tilde u_h\|_{L^2(\Gamma)}
     &\le c\,\|\Pi_{ad} \left(\frac1\alpha\,\bar y\, \bar p\right)
       - \Pi_{ad} \left(\frac1\alpha\,\bar y_h\, \bar p_h\right)\|_{L^2(\Gamma)} \\
     &\le \frac{c}\alpha\,\left(\|\bar y - \bar y_h\|_{L^2(\Gamma)}\,\|\bar p\|_{L^\infty(\Omega)}
       + \|\bar y_h\|_{L^\infty(\Omega)}\,\|\bar p - \bar p_h\|_{L^2(\Gamma)}\right).
   \end{align*}
   The assertion follows after insertion of \eqref{eq:postprocessing_basic} together with
   the estimates obtained in Lemmata \ref{lem:postprocessing_first}, \ref{lem:postprocessing_second}
   and \ref{lem:postprocessing_third}, as well as the stability estimates of $Z$ and $S_h$
   from Lemmata \ref{lem:props_S} and \ref{lem:stability_ShZh}, respectively.
 \end{proof}

 \section{Numerical experiments}\label{sec:experiments}
 
 It is the purpose of this last section to confirm the theoretical results
 by numerical experiments. To this end, we reformulate the discrete optimality condition
 \eqref{eq:deriv_jh} and use the equivalent projection formula
 \begin{equation}\label{eq:proj_formula_discrete}
   u_h = \Pi_{ad}\left(\frac1\alpha R_h^{\text{Simp}}(S_h(u_h)\,Z_h(u_h))\right).
 \end{equation}
 Here, $R_h^{\text{Simp}}\colon C(\Gamma)\to U_h$ is a projection operator based on the Simpson rule,
 this is, 
 \begin{equation*}
   [R_h^{\text{Simp}}(v)]_{E} = \frac16\left(v(x_{E_1}) + 4v(x_{E}) + v(x_{E_2})\right),
 \end{equation*}
 where $x_{E_1}$ and $x_{E_2}$ are the endpoints of the boundary edge $E\in\mathcal E_h$
 and $x_E$ its midpoint.
 The numerical solution of \eqref{eq:proj_formula_discrete} is computed by a semismooth Newton-method.

 The input data of the considered benchmark problem is chosen as follows.
 The computational domain is the unit square $\Omega:=(0,1)^2$.
 We define the exact Robin parameter $\tilde u$ by
 \[
   \tilde u(x_1,x_2):=
   \begin{cases}
     \max(-0.01,\ 1-30(x_1-0.5)^2), &\mbox{if}\ x_1=0, \\
     -0.01, &\mbox{otherwise},
   \end{cases}
 \]
 and use the desired state $y_d = S_h(\tilde u)$
 and the right-hand side $f\equiv 0$. Moreover, the regularization
 parameter $\alpha = 10^{-2}$ and the control bounds $u_a=0$, $u_b=\infty$ are used.
 
 We compute the numerical solution of our benchmark problem on a sequence of meshes
 starting with $\mathcal T_{h_0}$, $h_0=\sqrt{2}$, 
 consisting of two rectangular triangles only. The remaining grids $\mathcal T_{h_i}$,
 $i=1,2,\ldots,$ are obtained by a double bisection through the longest edge of each element
 applied to the previous mesh. This guarantees
 $h_i = \frac12 h_{i-1}$. In order to compute the discretization error we use the solution
 on the mesh $\mathcal T_{h_{11}}$ as an approximation of the exact solution, this means,
 \begin{equation*}
   \|\bar u-\bar u_{h_i}\|_{L^2(\Gamma)} \approx \|\bar u_{h_{11}} - \bar u_{h_i}\|_{L^2(\Gamma)},\quad i=0,1,\ldots,10.
 \end{equation*}
 Analogously, we compute the error for the approximation obtained by the postprocessing strategy.
 However, in this case the exact solution is approximated by
 $\bar u\approx \Pi_{ad}(\frac1\alpha \bar y_{h_{11}}\,\bar p_{h_{11}})$. The error norms
 $\|\Pi_{ad}(\frac1\alpha \bar y_{h_{11}}\,\bar p_{h_{11}}) -
 \Pi_{ad}(\frac1\alpha \bar y_{h_{i}}\,\bar p_{h_{i}})\|_{L^2(\Gamma)}$, $i=0,\ldots,11$, are computed element-wise by the Simpson quadrature formula
 with the modification that elements $E$ are split at those points where
 $\bar y_{h_i}\,\bar p_{h_i}$ or $\bar y_{h_{11}}\,\bar p_{h_{11}}$ change its sign.

 \begin{figure}
   \begin{center}
     \includegraphics[width=.8\textwidth]{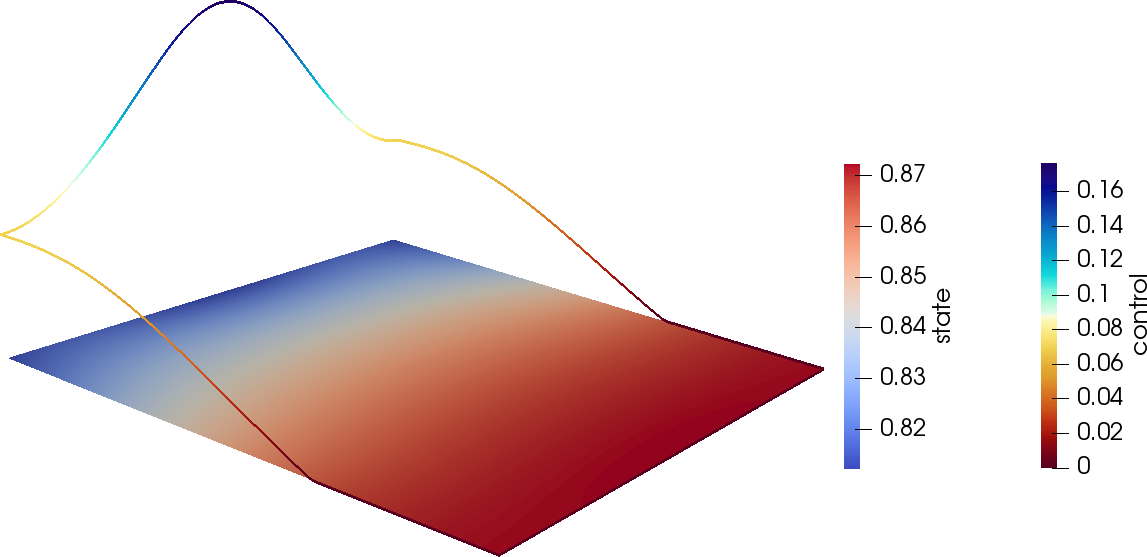}     
   \end{center}
   \caption{Optimal state (surface) and the optimal control (boundary curve) for
     the benchmark problem.}
   \label{fig:control_state}
 \end{figure}
 The optimal control and corresponding state of our benchmark problem is illustrated in
 Figure~\ref{fig:control_state} and the measured discretization errors as well as the experimentally
 computed convergence rates are summarized in Table \ref{tab:experiment}.
 As we have proven in Theorem \ref{thm:convergence_full_disc} the numerical solutions obtained by a
 full discretization using a piecewise constant control approximation converge with the optimal
 convergence rate $1$. Moreover, it is confirmed that the solution obtained with a
 postprocessing step, see Theorem \ref{thm:postprocessing}, converges with order $2$.
 Note that we actually proved the results for the case that the boundary is smooth which is
 indeed not the case in our example. However, the corner singularities contained in the solution are
 for a $90^\circ$-corner comparatively mild so that the regularity results from Lemma
 \ref{lem:regularity_improved} remain valid.
 \begin{table}[htb]
   \begin{center}
     \begin{tabular}{lrrrr}
       \toprule
       \multicolumn{1}{l}{$i$} & \multicolumn{1}{l}{DOF} & \multicolumn{1}{l}{BD DOF} & \multicolumn{1}{l}{$\|u-u_{h_i}\|_{L^2(\Gamma)}$ (eoc)} & \multicolumn{1}{l}{$\|u-\tilde u_{h_i}\|_{L^2(\Gamma)}$ (eoc)} \\  \midrule
       $3$ & 113 & 32 & 1.60e-2 (1.06) & 1.81e-2 (1.15) \\ 
       $4$ & 353 & 64 & 5.81e-3 (1.46) & 4.43e-3 (2.03) \\ 
       $5$ & 1217 & 128 & 2.56e-3 (1.18) & 1.03e-3 (2.11) \\ 
       $6$ & 4481 & 256 & 1.24e-3 (1.05) & 1.65e-4 (2.64) \\ 
       $7$ & 17153 & 512 & 6.17e-4 (1.00) & 7.52e-5 (1.13) \\ 
       $8$ & 67073 & 1024 & 3.06e-4 (1.01) & 1.79e-5 (2.07) \\ 
       $9$ & 265217 & 2048 & 1.49e-4 (1.04) & 4.31e-6 (2.05) \\ 
       $10$ & 1054716 & 4096 & 6.67e-5 (1.16) & 8.50e-7 (2.34) \\
       \bottomrule
     \end{tabular}
     \caption{Experimentally computed errors for the full discretization and the postprocessing approach with the corresponding experimental convergence rates (in parentheses)}
     \label{tab:experiment}
   \end{center}
\end{table}

\begin{appendix}
  \section{Proof of Theorem \ref{thm:fe_error_linfty}}
  \label{app:maximum_norm_estimate}
  The proof of the maximum norm estimate presented in Theorem \ref{thm:fe_error_linfty}
  follows basically from the arguments of \cite{FR76,Sco76}.
  For the convenience of the reader we want to repeat the proof 
  as the result of Theorem~\ref{thm:fe_error_linfty} is, for our specific situation, not
  directly available in the literature. The novelty of the present proof is that
  it includes curved elements as well as Robin boundary conditions.
  In the aforementioned articles, a representation of
  the error term based on a regularized Dirac function is used. This function forms
  the right-hand side of a dual problem whose solution is an approximation of 
  Green's function. The main difficulty is to bound this solution in appropriate norms.

  To this end, we denote by $T^*\in\mathcal T_h$ the element where $|y-y_h|$ attains its maximum.
  The regularized Dirac function is defined by
  $\delta^h(x) := |T^*|^{-1}\sgn(y(x)-y_h(x))$ if $x\in T^*$, and $\delta^h(x):=0$ if $x\not\in T^*$.
  The corresponding Green's function denoted by $g^h$ solves the problem
  \begin{equation}\label{eq:reg_green}
    -\Delta g^h + g^h = \delta^h \quad\mbox{in}\quad\Omega,\qquad \partial_n g^h + u\,g^h = 0\quad\mbox{on}\quad\Gamma.
  \end{equation}
  The Dirac function satisfies the properties
  \begin{equation}\label{eq:prop_dirac}
    \|\delta^h\|_{L^1(\Omega)} \le c,\qquad \|\delta^h\|_{L^2(\Omega)} \le c\,h^{-1}.
  \end{equation}
  
  We start our considerations with some a priori estimates for the solution $g^h$.
  \begin{lemma}\label{lem:prop_gh}
    The following a priori estimates hold:
    \begin{align*}
      &(i)&  \|g^h\|_{H^1(\Omega)} &\le c\,\lnh^{1/2}   &(ii)&& \|g^h\|_{H^2(\Omega)} &\le c\,h^{-1} \\
      &(iii)& \|g^h\|_{L^\infty(\Omega)} &\le c\,\lnh &&&&
    \end{align*}
  \end{lemma}
  \begin{proof}
    (ii) To show the estimate in the $H^2(\Omega)$-norm we apply the a priori estimate from
    Lemma \ref{lem:props_S}c) and $\|\delta^h\|_{L^2(\Omega)} \le c h^{-1}$.\\[.5em]    
    (i) The weak form of \eqref{eq:reg_green} and the property \eqref{eq:prop_dirac} imply 
    \begin{align}\label{eq:H1_to_Linfty}
      \gamma_u\,\|g^h\|_{H^1(\Omega)}^2
      &\le a_u(g^h,g^h) = (\delta^h,g^h)_{L^2(\Omega)} \le c\,\|g^h\|_{L^\infty(\Omega)}
        \nonumber\\
      &\le c \left(\|g^h-g_h^h\|_{L^\infty(\Omega)} + \lnh^{1/2}\,\|g_h^h\|_{H^1(\Omega)}\right),
    \end{align} 
    where the discrete Sobolev inequality was applied in the last step.
    The function $g_h^h\in V_h$ is the Ritz-projection of $g^h$ and
    satisfies the usual stability estimate
    \begin{equation}\label{eq:linfty_gh}
      \|g_h^h\|_{H^1(\Omega)} \le c\,\|g^h\|_{H^1(\Omega)}.
    \end{equation}
    Next, we derive a suboptimal error estimate for the finite-element error in the $L^\infty(\Omega)$-norm.
    Using an inverse inequality, estimates for the interpolant $\Pi_h$ from \eqref{eq:int_error},
    the Aubin-Nitsche trick and the a priori estimate shown already in the $H^2(\Omega)$-norm 
    we deduce 
    \begin{align}\label{eq:linfty_gh-ghh}
      \|g^h-g^h_h\|_{L^\infty(\Omega)}
      &\le \|g^h - \Pi_h g^h\|_{L^\infty(\Omega)}
        + c\,h^{-1}\left(\|g^h - \Pi_h g^h\|_{L^2(\Omega)} + \|g^h - g_h^h\|_{L^2(\Omega)}\right)\nonumber\\
      &\le c\,h\,\|g^h\|_{H^2(\Omega)} \le c.
    \end{align}
    Note that we hide the dependency on $u$, or more precisely on $\|u\|_{H^{1/2}(\Omega)}$
    and lower-order norms, in the generic constant to simplify the notation.
    Insertion of \eqref{eq:linfty_gh} and \eqref{eq:linfty_gh-ghh} into
    \eqref{eq:H1_to_Linfty}
    yields with Young's inequality
    \begin{equation*}
      \gamma_u\,\|g^h\|_{H^1(\Omega)}^2 \le c\,\lnh + \frac12\,\gamma_u\,\|g^h\|_{H^1(\Omega)}^2.
    \end{equation*}
    The desired estimate follows form a kick-back-argument.\\[.5em]    
    (iii) The $L^\infty(\Omega)$-estimate follows directly from \eqref{eq:H1_to_Linfty},
    \eqref{eq:linfty_gh} and \eqref{eq:linfty_gh-ghh} using the inequality (i).
  \end{proof}
  Next, we show an a priori estimate for $g^h$ in a weighted norm.
  This is the key idea which allows us to bound second derivatives
  by a logarithmic factor only. The weight function we will use is defined by
  \begin{equation*}
    \sigma(x) := \sqrt{|x-x^*|_2^2 + c\,h^2},
  \end{equation*}
  with $x^*:=\argmax_{x\in T^*} |y-y_h|(x)$. This function satisfies
  \begin{equation}\label{eq:props_sigma}
    \|\sigma^{-1}\|_{L^2(\Omega)}\le c\,\lnh^{1/2},\qquad \|\sigma^{-1}\|_{L^2(\Gamma)} \le c\,h^{-1/2}\,\lnh^{1/4}.
  \end{equation}
  The first property follows from a simple computation.
  The second estimate is a consequence of the multiplicative trace theorem
  $
  \|\sigma^{-1}\|_{L^2(\Gamma)} \le c\,\|\sigma^{-1}\|_{L^2(\Omega)}^{1/2}\,\|\sigma^{-1}\|_{H^1(\Omega)}^{1/2}$, see \cite[Theorem~1.6.6]{BS02},
  and the chain rule $\nabla \sigma^{-1} = -\sigma^{-2} \nabla \sigma$ and it is
  simple to show $\|\nabla\sigma\|_{L^\infty(\Omega)} \le 1$
  and $\|\sigma^{-2}\|_{L^2(\Omega)} \le c\,h^{-1}$.
  \begin{lemma}
    Assume that $u\in H^{1/2}(\Gamma)$. There holds the estimate
    \begin{equation*}
      \|\sigma\,\nabla^2 g^h \|_{L^2(\Omega)} \le c\,\lnh^{1/2}.
    \end{equation*}
  \end{lemma}
  \begin{proof}
    We introduce the functions $\xi_i:=|x_i-x_i^*|$, $i=1,2$, which allow us to write
    \begin{equation}\label{eq:max_est_weighted_H2}
      \|\sigma\,\nabla^2 g^h\|_{L^2(\Omega)}^2 = \sum_{i=1}^2\|\xi_i\,\nabla^2 g^h\|_{L^2(\Omega)}^2
      + c\,h^2\,\|\nabla^2 g^h\|_{L^2(\Omega)}^2.
    \end{equation}
    With the reverse product rule we obtain
    \begin{equation}\label{eq:H2_reg_gh_1}
      \|\xi_i\,\nabla^2 g^h\|_{L^2(\Omega)}^2 \le \|\nabla^2(\xi_i\,g^h)\|_{L^2(\Omega)}^2 + \|\nabla g^h \|_{L^2(\Omega)}^2.
    \end{equation}
    Moreover, we easily confirm that $\xi_i\,g^h$ is the solution of the problem
    \begin{equation*}
    \begin{aligned}
      -\Delta (\xi_i\,g^h) + \xi_i\, g^h &= -2\,\frac{\partial g^h}{\partial x_i} + \xi_i\,\delta^h &\mbox{in}\ \Omega,\\
      \partial_n (\xi_i\,g^h) + u\,\xi_i\,g^h &=  g^h\,n_i &\mbox{on}\ \Gamma,
    \end{aligned}
  \end{equation*}
  where $n_i$ is the $i$-th component of the outer unit normal vector on $\Gamma$.
  Lemma \ref{lem:props_S}c) using the property $ \|\xi_i\,\delta^h\|_{L^2(\Omega)} \le c$,
  which follows from a simple computation, leads to
    \begin{align}\label{eq:max_est_H2_aux}
      \|\nabla^2(\xi_i\,g^h)\|_{L^2(\Omega)}
      &\le c \left(1+\|u\|_{H^{1/2}(\Gamma)}\right)
        \left(1+\|\nabla g^h\|_{L^2(\Omega)} 
        + \|g^h\|_{H^{1/2}(\Gamma)}\right)\nonumber\\
      &\le c\left(1+\|g^h\|_{H^1(\Omega)}\right).
    \end{align}
    Insertion into \eqref{eq:H2_reg_gh_1} and using Lemma \ref{lem:prop_gh}(i) leads to
    \begin{equation*}
      \|\xi_i\,\nabla^2 g^h\|_{L^2(\Omega)} \le c\,\lnh^{1/2}, \quad i=1,2.
    \end{equation*}
    An estimate for the second term on the right-hand side of \eqref{eq:max_est_weighted_H2} 
    is derived in Lemma \ref{lem:prop_gh}(ii).
  \end{proof}
  Next, we derive some error estimates for the approximation $g_h^h$ in several norms.
  \begin{lemma}\label{lem:error_gh-ghh}
    Assume that $u\in H^{1/2}(\Gamma)$. Then, there hold the error estimates
    \begin{equation*}
      \begin{aligned}
        h^{-1}\,\|g^h-g_h^h\|_{L^2(\Omega)} + \|\nabla(g^h-g_h^h)\|_{L^2(\Omega)}
        &\le c,\\
        \|\sigma\nabla^2 (g^h-g_h^h)\|_{L^2_{\textup{pw}}(\Omega)} &\le c\,\lnh^{1/2}.
      \end{aligned}
    \end{equation*}
  \end{lemma}
  \begin{proof}
    The first estimate follows directly from the $H^1(\Omega)$-error estimate
    stated in Lemma \ref{lem:h1_l2_error} and the Aubin-Nitsche trick.
    Moreover, the a priori estimate for the $H^2(\Omega)$-norm of $g^h$
    from Lemma \ref{lem:prop_gh} has to be exploited.

    In the second estimate one observes that the discrete function $g_h^h$ would vanish
    except on curved elements (note that $g_h^h$ is affine on the reference element only, but not on $T$). With the transformation result \cite[Lemma 2.3]{Ber89} we obtain
    \begin{align*}
      \|\nabla^2(g^h - g_h^h)\|_{L^2(T)}
      &\le\,c\,|T|^{1/2}\,h^{-2}\sum_{k=0}^2 h^{4-2k}\,
        \|\hat\nabla^k(\hat g^h - \hat g_h^h)\|_{L^2(\hat T)} \\
      &\le c\left(\|\nabla^2 g^h\|_{L^2(T)} + h\,\|g^h - g_h^h\|_{H^1(T)}\right),
    \end{align*}
    where $g^h = \hat g^h\circ F_T^{-1}$, $g_h^h=\hat g_h^h\circ F_T^{-1}$.
    Taking into account $\inf_{x\in T}\sigma(x) \sim \sup_{x\in T} \sigma(x)$, which holds due to the assumed shape-regularity, and $|\sigma(x)| \le c$ for all $x\in\Omega$, we obtain
    \begin{equation*}
      \|\sigma\,\nabla^2 (g^h-g_h^h)\|_{L^2_{\textup{pw}}(\Omega)}
      \le c \left(\|\sigma\,\nabla^2 g^h\|_{L^2(\Omega)} + h\,\|g^h - g_h^h\|_{H^1(\Omega)}\right).
    \end{equation*}
    The first term has been discussed in the previous Lemma and the last term
    has been considered in the present Lemma already.
  \end{proof}

  Now we are in the position to prove Theorem \ref{thm:fe_error_linfty}.\\
  \begin{proof}
    With an inverse inequality and the H\"older inequality, the definition of $\delta^h$
    and a maximum norm estimate for the interpolant $\Pi_h$, see e.\,g.\ \cite[Theorem 4.1]{Ber89},
    we obtain
    \begin{align}\label{eq:max_est_start}
      \|y-y_h\|_{L^\infty(\Omega)} = \|y-y_h\|_{L^\infty(T^*)}
      &\le c \left(\|y-\Pi_h y\|_{L^\infty(T^*)} + |T^*|^{-1}\,\|\Pi_h y-y_h\|_{L^1(T^*)}\right)\nonumber\\
      &\le c \left(\|y-\Pi_h y\|_{L^\infty(T^*)} + (\delta^h,y-y_h)_{L^2(\Omega)}\right)\nonumber\\
      &\le c \left(h^{2-2/q}\,\|y\|_{W^{2,q}(\Omega)} + a_u(y-y_h,g^h)\right),
    \end{align}
    where $g_h^h\in V_h$ denotes the Ritz projection of $g^h$.
    
    For the latter part on the right-hand side of \eqref{eq:max_est_start}
    we get with the Galerkin orthogonality, the H\"older inequality,
    a trace theorem for the boundary integral term
    as well as $\|u\|_{L^\infty(\Gamma)} \le c$
    \begin{align}\label{eq:max_est_orthogonality}
      & a_u(y-y_h,g^h) = a_u(y-\Pi_h y,g^h- g_h^h)\nonumber\\
      &\quad \le c\,\|y-\Pi_hy\|_{W^{1,\infty}(\Omega)}\,\|g^h-g_h^h\|_{W^{1,1}(\Omega)}.
    \end{align}
    An estimate for the interpolation error is deduced 
    in \cite{Ber89}.
    The $L^1(\Omega)$-norms can be replaced by weighted $L^2(\Omega)$-norms
    involving the weighting function $\sigma$. Taking into account the
    properties \eqref{eq:props_sigma} we obtain
    \begin{equation}\label{eq:max_est_L1_est_1}
      \|g^h-g_h^h\|_{W^{1,1}(\Omega)} 
      \le c\,\lnh^{1/2}\, \left(\|\sigma\,\nabla (g^h-g_h^h)\|_{L^2(\Omega)}
        + \|g^h-g_h^h\|_{L^2(\Omega)}\right).
    \end{equation}
    In the following we will show that the expressions on the right-hand side
    of \eqref{eq:max_est_L1_est_1} are bounded by $c\,h\,\lnh^{1/2}$.
    Therefore, we apply the reverse product rule and get
    \begin{align*}
      &\|\sigma\, \nabla (g^h-g_h^h)\|_{L^2(\Omega)}^2\\
      &\qquad = (\nabla(\sigma^2\, (g^h -g_h^h)), \nabla(g^h-g_h^h))_{L^2(\Omega)}
        - ((g^h-g_h^h)\, \nabla\sigma^2, \nabla(g^h-g_h^h))_{L^2(\Omega)}.
    \end{align*}
    From this we conclude 
    \begin{align}\label{eq:max_est_milestone}
      & \Theta^2:= \|\sigma \nabla (g^h-g_h^h)\|_{L^2(\Omega)}^2  + \|g^h-g_h^h\|_{L^2(\Omega)}^2 \nonumber\\
      &\qquad \le a_u(\sigma^2(g^h-g_h^h),g^h-g_h^h) - ((g^h-g_h^h) \nabla\sigma^2, \nabla(g^h-g_h^h))_{L^2(\Omega)}.
    \end{align}
    Here, we exploited that $(u\,\sigma^2\,(g^h-g_h^h),g^h-g_h^h)_{L^2(\Gamma)}\ge 0$
    due to $u\ge u_a\ge 0$.
    Next, we introduce the abbreviation $z:=\sigma^2\,(g^h-g_h^h)$. The
    Galerkin orthogonality of $g^h-g_h^h$, Young's inequality and the trace
    theorem taking into account $|\sigma| + |\nabla \sigma|\le c$ yield
    \begin{equation*}
      \|\sigma\,v\|_{L^2(\Gamma)} \le c\left(\|\sigma\,v\|_{L^2(\Omega)} +
      \|\nabla(\sigma\,v)\|_{L^2(\Omega)}\right)
      \le c\left(\|v\|_{L^2(\Omega)} + \|\sigma\,\nabla v\|_{L^2(\Omega)}\right)
    \end{equation*}
    and thus,
    \begin{align}\label{eq:max_est_interpolation_start}
      &a_u(\sigma^2(g^h-g_h^h),g^h-g_h^h) = a_u(z-\Pi_h z,g^h-g_h^h) \nonumber\\
      &\qquad\le \frac14\, \Theta^2 + c\,\Big(\|\sigma^{-1}\,\nabla(z-\Pi_h
        z)\|_{L^2(\Omega)}^2  \nonumber\\
      &\qquad \phantom{\frac12 \Theta^2 + c\le}+\|\sigma^{-1}\,(z-\Pi_h
        z)\|_{L^2(\Omega)}^2 +
        \|\sigma^{-1}\,u\,(z-\Pi_h z)\|_{L^2(\Gamma)}^2\Big).
    \end{align}
    Next, we derive local interpolation error estimates. 
    In the following we use the notation $\underline\sigma_T:=\inf_{x\in T} \sigma(x)$ and $\overline\sigma_T :=\sup_{x\in T} \sigma(x)$. Due to the assumed shape-regularity there holds 
    $\underline\sigma_T\sim\overline\sigma_T$ for all $T\in \mathcal T_h$, and hence,
    \begin{equation*}
      \|\sigma^{-1}\,\nabla(z-\Pi_h z)\|_{L^2(T)} + h^{-1}\,\|\sigma^{-1}\,(z-\Pi_h z)\|_{L^2(T)}
      \le c\,\underline\sigma_T^{-1}\, h\,
      \|\sigma^2\,(g^h-g_h^h)\|_{H^2(\tilde T)},
    \end{equation*}
    where $\tilde T$ is the patch of all elements adjacent to $T$ (note that
    $\Pi_h$ is a quasi-interpolant).
    The Leibniz rule and the properties $|\nabla\sigma^2|\le\sigma$ and
    $|\nabla^2 \sigma^2|\le c$ imply
    \begin{equation*}
      \|\sigma^2\,(g^h-g_h^h)\|_{H^2(\tilde T)} \le \|g^h-g_h^h\|_{L^2(\tilde T)}
      + \|\sigma\,\nabla(g^h-g_h^h)\|_{L^2(\tilde T)} + \|\sigma^2\,\nabla^2
      (g^h-g_h^h)\|_{L^2(\tilde T)}.
    \end{equation*}
    Next, we combine the two estimates above and take into account the properties
    $h\,\underline\sigma_T^{-1} \le c$ and $\overline \sigma_T\sim \overline\sigma_{\tilde T}$
    which follows from the assumed quasi-uniformity.
    Summation over all $T\in\mathcal T_h$
    and an application of Lemma \ref{lem:error_gh-ghh} yields
    \begin{align}\label{eq:max_est_interpolation_global}
      &\|\sigma^{-1}\nabla(z-\Pi_h z)\|_{L^2(\Omega)} + h^{-1}\,\|\sigma^{-1}(z-\Pi_h z)\|_{L^2(\Omega)}\nonumber\\
      &\quad \le c \left(\|g^h-g_h^h\|_{L^2(\Omega)} + h\,\|\nabla(g^h-g_h^h)\|_{L^2(\Omega)} + h\,\|\sigma\,\nabla^2 (g^h-g_h^h)\|_{L^2_{\textup{pw}}(\Omega)}\right) \nonumber\\
      &\quad\le c\,h\,\lnh^{1/2}.
    \end{align} 

    It remains to discuss the third term on the right-hand side of
    \eqref{eq:max_est_interpolation_start}.
    With interpolation error estimates for $\Pi_h$ on the boundary, compare
    also \eqref{eq:int_error_boundary}, and $u\in L^\infty(\Gamma)$ we obtain
    \begin{align}\label{eq:max_est_3rd_term}
      \|\sigma^{-1}\,u\,(z-\Pi_h z)\|_{L^2(E)}
      &\le c\,h\,\underline\sigma_E^{-1} \|\nabla z\|_{L^2(\tilde E)}\nonumber\\
      &\le c\,h\left(\|g^h-g_h^h\|_{L^2(E)} +
        \|\sigma\,\nabla(g^h-g_h^h)\|_{L^2(\tilde E)}\right),
    \end{align}
    where we exploited the product rule and the property $\nabla\sigma^2\le 2\sigma\vec1$
    in the last step.
    With a trace theorem and Lemma \ref{lem:error_gh-ghh} we conclude
    \begin{equation*}
      \|g^h-g_h^h\|_{L^2(\Gamma)} \le c\,\|g^h-g_h^h\|_{H^1(\Omega)} \le c,
    \end{equation*}
    and with a multiplicative trace theorem, Young's inequality, the product
    rule and the estimates from Lemma \ref{lem:error_gh-ghh} we obtain
    \begin{align*}
      \|\sigma\,\nabla(g^h-g_h^h)\|_{L^2(\Gamma)}
      &\le c \left(\|\sigma\,\nabla(g^h-g_h^h)\|_{L^2(\Omega)} + \|\nabla
        (\sigma\,\nabla(g^h-g_h^h))\|_{L_{\textup{pw}}^2(\Omega)} \right)\\
      &\le c \left(\|\nabla(g^h-g_h^h))\|_{L^2(\Omega)} + \|\sigma\,\nabla^2(g^h-g_h^h)\|_{L_{\textup{pw}}^2(\Omega)}\right)\\
      &\le c\,\lnh^{1/2}.
    \end{align*}
    The estimate \eqref{eq:max_est_3rd_term} then simplifies to
    \begin{equation}\label{eq:max_est_3rd_term_final}
      \|\sigma^{-1}\,(z-\Pi_h z)\|_{L^2(\Gamma)}
      \le c\,h\,\lnh^{1/2}.
    \end{equation}

    Insertion of \eqref{eq:max_est_interpolation_global} and \eqref{eq:max_est_3rd_term_final} into \eqref{eq:max_est_interpolation_start} leads to the estimate
    \begin{equation}\label{eq:max_est_part1}
      a(\sigma^2(g^h-g_h^h),g^h-g_h^h) \le \frac14 \Theta^2 + c\,h^2\,\lnh.
    \end{equation}
    It remains to show an estimate for the second term on the right-hand side 
    of \eqref{eq:max_est_milestone}.
    Due to $|\nabla\sigma^2| \le 2\sigma\vec 1$, Young's inequality
    and the $L^2(\Omega)$-error estimate from Lemma \ref{lem:error_gh-ghh}
    we get
    \begin{align}\label{eq:max_est_part2}
      ((g^h-g_h^h) \nabla\sigma^2, \nabla(g^h-g_h^h))_{L^2(\Omega)}
      &\le c\,\|g^h-g_h^h\|_{L^2(\Omega)}^2 + \frac14\,\|\sigma\,\nabla(g^h-g_h^h)\|_{L^2(\Omega)}^2\nonumber\\
      &\le c\,h^2 + \frac14\,\Theta^2.
    \end{align}
    Insertion of \eqref{eq:max_est_part1} and \eqref{eq:max_est_part2} into
    \eqref{eq:max_est_milestone}
    yields
    \begin{equation}\label{eq:final_est_theta}
      \Theta^2 \le \frac12\,\Theta^2 + c\,h^2\,\lnh
    \end{equation}
    and with a kick-back-argument we conclude $\Theta^2=c\,h^2\,\lnh$.
    Finally, we collect up the previous estimates. To this end, we insert
    \eqref{eq:final_est_theta} into \eqref{eq:max_est_L1_est_1}, the resulting estimate into
    \eqref{eq:max_est_orthogonality} and this into \eqref{eq:max_est_start}.
  \end{proof}
  \section{Local estimates for the midpoint interpolant and the $L^2(\Gamma)$-projection}\label{app:midpoint}
  To the best of the author's knowledge there are no error estimates for the midpoint interpolant defined on a curved boundary available in the literature. Thus, we prove the following Lemmata which are needed in the proof of Lemma \ref{lem:postprocessing_second}.

  Consider a single boundary element $E\subset \bar T$ with corresponding element $T\in\mathcal T_h$.
  A parametrization of the boundary element is given by $E:=\{\gamma_E(\xi) := F_T(\xi,0),\ \xi\in(0,1)\}$ when assuming that the edge of $\hat T$ with endpoints $(0,0)$, $(1,0)$ is mapped onto $E$.
  In the following we denote the length of a boundary element $E\in\mathcal E_h$
  by $L_E = \int_0^1 |\dot\gamma_E(\xi)|\,\mathrm d\xi$.

  \begin{lemma}\label{lem:local_est_midpoint}
    For each function $u\colon \Gamma\to\mathbb R$ 
    there exists some piecewise constant function $R_h u\in U_h$
    satisfying the local estimates
    \begin{align*}
      \left\vert\int_E (u - R_h u)\right\vert &\le c\,h^{5/2}\left(\|\nabla u\|_{L^2(E)} + \|\nabla^2 u\|_{L^2(E)}\right),\\
      \|u-R_h u\|_{L^\infty(E)} &\le c\,h^{1-1/q}\,\|\nabla u\|_{L^q(E)},\quad
                                  q\in (1,\infty],
    \end{align*}
    for all $E\in\mathcal E_h$, provided that $u$ possesses the regularity demanded by the right-hand
    side.
  \end{lemma}
  \begin{proof}
    Let us first construct a suitable interpolation operator. To obtain the desired second-order accuracy we have to guarantee that the property $\int_E p = \int_E R_h p$ holds for all functions $p(\gamma_E(\xi)) = \hat p(\xi)$ with some first-order polynomial $\hat p(\xi):=a +b\,\xi$. The transformation to $\hat E:=(0,1)\times \{0\}$ yields
    \begin{equation*}
      \int_E(p(x)-R_h p) \mathrm ds_x = \int_0^1 (\hat p(\xi) - \hat p(\xi_E))
      \,|\dot\gamma_E(\xi)|\,\mathrm d\xi= b\,\int_0^1(\xi - \xi_E)\,|\dot\gamma_E(\xi)|\,\mathrm d\xi = 0.
    \end{equation*}
    The latter step holds true when choosing
    \begin{equation*}
      \xi_E := \frac{1}{L_E}\int_0^1 \xi\,|\dot\gamma_E(\xi)|\,\mathrm d\xi 
    \end{equation*}
    To this end, we define our operator by means of $R_h u|_E= \hat u(\xi_E)$.
    Obviously, the definition of $R_h$ depends on the transformations $F_T$.

    To show the interpolation error estimates we apply the property $\int_E(p-R_h p)=0$ for arbitrary
    $p\in\mathcal P_1$, the stability of the interpolant $\hat R_h \hat u|_E = \hat u(\xi_E)$,
    the properties \eqref{eq:props_trafo} of the transformation $F_T$ and the
    Bramble-Hilbert Lemma. This yields
    \begin{align*}
      \int_E(u(x)-R_h u)\mathrm ds_x
      &= \int_0^1(I-\hat R_h)(\hat u - \hat p)(\xi)\,|\dot\gamma_E(\xi)|\,\mathrm d\xi \\
      &\le c\,h\,\|\hat u - \hat p\|_{L^\infty(\hat E)}
        \le c\,h\,\|\partial_{\xi\xi}\hat u\|_{L^2(\hat E)}.
    \end{align*}
    For the transformation back to the world element $E$ we apply the chain
    rule \[\partial_{\xi\xi} \hat u(\xi) = \dot\gamma_E(\xi)^\top\, \nabla^2
    u(\gamma_E(\xi)) \,\dot\gamma_E(\xi) + \ddot\gamma_E(\xi)^\top\,\nabla u(\gamma_E(\xi))
    \]
  and the properties \eqref{eq:props_trafo} to arrive at
    \begin{align*}
      \|\partial_{\xi\xi}\hat u\|_{L^2(\hat E)}
      & \le c\,h^2 \left(\max_{\xi}|\dot\gamma_E(\xi)|^{-1} \sum_{|\alpha|=1}^2\int_0^1 (D^\alpha u(\gamma_E(\xi)))^2\,|\dot\gamma_E(\xi)|\,\mathrm d\xi \right)^{1/2} \\
      & \le c\,h^2 \left(\min_{\xi}|\dot\gamma_E(\xi)|\right)^{-1/2} \left(\|\nabla u\|_{L^2(E)} + \|\nabla^2 u\|_{L^2(E)}\right).
    \end{align*}
    Finally, the norm of $\dot\gamma_E$ can be bounded by means of
    \begin{equation}\label{eq:est_gamma_-1}
      |\dot\gamma_E(\xi)|^{-1} = |DF_T(\xi,0) (1,\,0)^\top|^{-1} \le \left(\min_{\|x\|=1} |DF_T(\xi,0) x|\right)^{-1}
      = \|DF_T(\xi,0)^{-1}\|.
    \end{equation}
    Note, that the last step is valid for the spectral norm only.
    An application of Lemma 2.2 from \cite{Ber89} which provides $\sup_{\hat x\in \hat T}\|DF_T(\hat x)^{-1}\|\le c h^{-1}$ leads to the first estimate.

    The second estimate follows with similar arguments. For an arbitrary constant $\hat p$ we
    then obtain 
    \begin{equation*}
      \|u-R_h u\|_{L^\infty(E)} \le \|(I-\hat R_h)(\hat u-\hat p)\|_{L^\infty(\hat E)}
      \le c\, |\hat u|_{W^{1,q}(\hat E)} \le c\, h\,|E|^{-1/q}\,\|\nabla u\|_{L^q(E)}.
    \end{equation*}
  \end{proof}
  A further operator that is needed in Section \ref{sec:optimal_control} is
  the $L^2(\Gamma)$-projection onto $U_h$.
  In case of curved boundaries, this operator
  reads
  \begin{equation*}
    [Q_h v]|_E := \frac1{L_E}\int_0^1 v(\gamma_E(\xi))\,|\dot\gamma_E(\xi)|\,\mathrm d\xi
  \end{equation*}
  for each $E\in\mathcal E_h$.
  Note that this definition implies the orthogonality property
  \begin{equation}\label{eq:orthogonality}
    (u-Q_h u, v_h)_{L^2(\Gamma)} = 0 \qquad\forall v_h\in U_h.
  \end{equation}
  With similar arguments as in the previous Lemma we obtain the following local estimate
  which is standard in case of a boundary consisting of straight edges.
  \begin{lemma}\label{lem:local_est_l2_proj}
    Assume that $u\in H^1(\Gamma)$. Then the estimate
    \begin{equation*}
      \|u-Q_h u\|_{L^2(E)} \le c\,h\,\|\nabla u\|_{L^2(E)}
    \end{equation*}
    is fulfilled for all $E\in\mathcal E_h$.
  \end{lemma}
  \begin{proof}
    We introduce a further projection onto $U_h$, namely
    $[\tilde Q_h u]|_E:=\int_0^1 u(\gamma_E(\xi))\,\mathrm d\xi$.
    Using \eqref{eq:orthogonality}, the transformation to the reference element as in the previous Lemma and the Poincar\'e inequality we obtain
    \begin{align*}
      \|u-Q_h u\|_{L^2(E)}^2
      &\le \|u-\tilde Q_h u\|_{L^2(E)}^2 \\
      & = \int_0^1 \left(u(\gamma_E(\xi)) - \int_0^1 u(\gamma_E(\xi'))\,\mathrm d\xi'\right)^2
        |\dot\gamma_E(\xi)|\,\mathrm d\xi \\
      &\le c\,h\,\|\partial_\xi u(\gamma_E(\cdot))\|_{L^2(0,1)}^2 \le c\,h^2\,\|\nabla u\|_{L^2(E)}^2
    \end{align*}
    where the last step is a consequence of the chain rule
    $\partial_\xi u(\gamma_E(\xi)) = \nabla u(\gamma_E(\xi))\cdot\dot\gamma_E(\xi)$ and
    $|\dot\gamma_E(\xi)|\sim h$ for all $\xi\in(0,1)$, see also \eqref{eq:est_gamma_-1}.
  \end{proof}

  We conclude this section with an estimate for an expression which is need in Lemma
  \ref{lem:postprocessing_second}.  
  \begin{lemma}\label{lem:int_error_Rh}
    Assume that the functions $u$ and $v$ belong to $\in H^1(\Gamma)$.
    Then the inequality
    \begin{equation*}
      (u-R_h u, v)_{L^2(\Gamma)}
      \le c\,h^2\,\|\nabla u\|_{L^2(\Gamma)}\,\|\nabla v\|_{L^2(\Gamma)} +
      c\,\|v\|_{L^\infty(\Gamma)} \sum_{E\in\mathcal E_h} \left\vert\int_E (u-R_h u)\right\vert 
    \end{equation*}
    is valid.      
  \end{lemma}
  \begin{proof}
    First, we split the term under consideration using the $L^2(\Gamma)$-projection onto
    $U_h$ and obtain
    \begin{equation*}
      (u-R_h u, v)_{L^2(\Gamma)} = (u-Q_h u, v - Q_h v)_{L^2(\Gamma)} + (Q_h (u-R_h u), v)_{L^2(\Gamma)}.
    \end{equation*} 
    The first term on the right-hand side can be treated with the local estimate from Lemma \ref{lem:local_est_l2_proj} which yields
    \begin{equation*}
      (u-Q_h u, v - Q_h v)_{L^2(\Gamma)}
      \le c\,h^2\, \|\nabla u\|_{L^2(\Gamma)}\,\|\nabla v\|_{L^2(\Gamma)}.
    \end{equation*}
    For the second term we exploit the definition of $Q_h$ and $R_h$ on the reference element.
    For each $E\in\mathcal E_h$ we then obtain
    \begin{align*}
      \|Q_h(u-R_h u)\|_{L^1(E)}
      &= \int_0^1 \left\vert  \frac1{L_E}\int_0^1(u(\gamma_E(\xi)) - u(\gamma_E(\xi_E)))\,
        |\dot\gamma_E(\xi)|\,\mathrm d\xi\right\vert\,|\dot\gamma_E(\xi')|\,\mathrm d\xi' \\
      &= \left\vert\int_0^1(u(\gamma_E(\xi))-u(\gamma_E(\xi_E)))\,|\dot\gamma_E(\xi)|\,\mathrm d\xi\right\vert
        = \left\vert\int_E (u-R_h u)\right\vert,
    \end{align*}
    where we used $\int_0^1|\dot\gamma_E(\xi')|\,\mathrm d\xi' = L_E$ in the second step.
    Consequently, we obtain
    \begin{equation*}
      (Q_h (u-R_h u), v)_{L^2(\Gamma)} \le c\,\|v\|_{L^\infty(\Gamma)}\sum_{E\in\mathcal E_h} \left\vert\int_E (u-R_h u)\right\vert 
    \end{equation*}
    and conclude the assertion.
  \end{proof}
  
\end{appendix}

\printbibliography

\end{document}